\documentclass[11 pt]{amsart}

\usepackage{color}
\usepackage{amssymb, amsmath, amsthm}
\usepackage{amsrefs}
\usepackage{fullpage}

\usepackage{graphicx}
\usepackage{pinlabel}
\usepackage{hyperref}

\newtheorem{theorem}{Theorem}[section]
\newtheorem*{theorem*}{Theorem}
\newtheorem{corollary}[theorem]{Corollary}

\newtheorem{lemma}[theorem]{Lemma}

\newtheorem*{main}{Main Theorem} 

\theoremstyle{definition}
\newtheorem{definition}[theorem]{Definition}
\newtheorem{example}[theorem]{Example}
\newtheorem{remark}[theorem]{Remark}

\newcommand{\R}{\mathbb{R}}

\newcommand{\ob}[1]{\overline{#1}}

\newcommand{\co}{\mskip0.5 mu\colon\thinspace} 
\newcommand{\nil}{\varnothing}

\newcommand{\up}{\uparrow}
\newcommand{\down}{\downarrow}
\newcommand{\wihat}[1]{\widehat{#1}}
\newcommand{\tild}[1]{\widetilde{#1}}

\newcommand{\defn}[1]{\textbf{#1}}

\newcommand{\vpH}{vp\mathbb{H}}
\newcommand{\vpoH}{\overrightarrow{\vpH}}

\newcommand{\punct}[1]{\mathring{#1}}

\newcommand{\more}{\to}

\newcommand{\oc}{\overrightarrow{\mathbf{c}}}

\newcommand{\boundary}{\partial}
\newcommand{\bdd}{\boundary}
\newcommand{\mc}[1]{\mathcal{#1}}

\begin{document}

\title{Thin position for knots, links, and graphs in 3-manifolds}
   \author{Scott A. Taylor, Maggy Tomova}
   
   \begin{abstract}
We define a new notion of thin position for a graph in a 3-manifold which combines the ideas of thin position for manifolds first originated by Scharlemann and Thompson with the ideas of thin position for knots first originated by Gabai. This thin position has the property that connect summing annuli and pairs-of-pants show up as thin levels. In a forthcoming paper, this new thin position allows us to define two new families of invariants of knots, links, and graphs in 3-manifolds. The invariants in one family are similar to bridge number and the invariants in the other family are similar to Gabai's width for knots in the 3-sphere. The invariants in both families detect the unknot and are additive under connected sum and trivalent vertex sum.
\end{abstract}
   
 \maketitle  
\section{Introduction}

In \cite{G3}, Gabai introduced \emph{width} as an extremely useful knot invariant. Width is a certain function (whose exact definition isn't needed for our purposes) from the set of height functions on a knot in $S^3$ to the natural numbers. It becomes an invariant after minimizing over all possible height functions. A particular height function of a knot is \emph{thin} if it realizes the minimum width. Thin embeddings produce very useful topological information about the knot (see, for example, Gabai's proof of Property R in \cite{G3}; Gordon and Luecke's solution to the knot complement problem \cite{GL}; and Thompson's proof \cite{Thompson} that small knots have thin position equal to bridge position.) Scharlemann and Thompson extended Gabai's width for knots to a width for graphs in $S^3$ \cite{ST-graphthin} and gave a new proof of Waldhausen's classification of Heegaard splittings of $S^3$ \cite{Wald}. They also applied  a similar idea to handle structures of 3-manifolds, producing an invariant of 3-manifolds also called width \cite{ST-thin}. A handle decomposition which attains the width is said to be \emph{thin}. Thin handle decompositions for 3-manifolds have been very useful for understanding the structure of Heegaard splittings of 3-manifolds.

There have been a number of attempts (eg. \cite{Bachman, HRS, HS01, Johnson, TT2}) to define width of knots (and later for tangles and graphs) in a 3-manifold by using various generalizations of the Scharlemann-Thompson constructions. These definitions have been used in various ways, however they have never been as useful as Scharlemann and Thompson's thin position for 3-manifolds. For instance, although Scharlemann and Thompson's thin position has the property that all thin surfaces in a thin handle decomposition for a closed 3-manifold are essential surfaces, the same is not true for the thin positions applied to knot and graph complements. (The papers \cite{Bachman} and \cite{TT2} are exceptions. The former, however, applies only to links in closed 3-manifolds and in the latter there are a number of technical requirements which limit its utility.) 

In this paper, we define an \emph{oriented multiple v.p.-bridge surface} as a certain type of surface $\mc{H}$ in a 3-manifold $M$ transverse to a graph $T \subset M$. The components of $\mc{H}$ are partitioned into \emph{thick surfaces} $\mc{H}^+$ and \emph{thin surfaces} $\mc{H}^-$. We exhibit a collection of thinning moves which give rise to a partial order, denoted $\more$, on the set of \emph{oriented multiple v.p.-bridge surfaces} (terms to be defined later) for a (3-manifold, graph) pair $(M,T)$. These thinning moves include the usual kinds of destabilization and untelescoping moves, known to experts, but we also include several new ones, corresponding to the situation when portions of the graph $T$ are cores of compressionbodies in a generalized Heegaard splitting of $M$ (in the sense of \cite{ST-thin}). More significantly, we also allow untelescoping using various generalizations of compressing discs. Throughout the paper, we show how these generalized compressing discs arise naturally when considering bridge surfaces for (3-manifold, graph) pairs. If $\mc{H}$ and $\mc{K}$ are oriented bridge surfaces, we say that $\mc{H} \more \mc{K}$ if certain kinds of carefully constructed sequences of thinning moves produce $\mc{K}$ from $\mc{H}$. If no such sequence can be applied to $\mc{H}$ then we say that $\mc{H}$ is \emph{locally thin}. If the reader allows us to defer some more definitions until later, we can state our results as:

\begin{main}
Let $M$ be a compact, orientable 3-manifold and $T \subset M$ a properly embedded graph such that no vertex has valence two and no component of $\boundary M$ is a sphere intersecting $T$ two or fewer times. Assume also that no sphere in $M$ intersects $T$ exactly once transversally. Then $\more$ is a partial order on $\vpoH(M,T)$. Furthermore, if $\mc{H} \in \vpoH(M,T)$, then there is a locally thin $\mc{K} \in \vpoH(M,T)$ such that $\mc{H} \more \mc{K}$. Additionally, if $\mc{H}$ is locally thin then the following hold:
\begin{enumerate}
\item Each component of $\mc{H}^+$ is sc-strongly irreducible in the complement of the thin surfaces.
\item No component of $(M, T)\setminus \mc{H}$ is a trivial product compressionbody between a thin surface and a thick surface.
\item Every component of $\mc{H}^-$ is c-essential in the exterior of $T$.
\item If there is a 2-sphere in $M$ intersecting $T$ three or fewer times and which is essential in the exterior of $T$, then some component of $\mc{H}^-$ is such a sphere.
\end{enumerate}
\end{main}

The properties of locally thin surfaces are proved in part by using sweep-out arguments. (See Theorems \ref{Properties Locally Thin} and \ref{thm:They are thin levels}.) The existence of a locally thin $\mc{K}$ with $\mc{H} \more \mc{K}$ is proved using a new complexity which decreases under thinning sequences. (See Theorem \ref{partial order}.) Although this complexity behaves much as Gabai's or Scharlemann-Thompson's widths do, we view it as being more like the complexities used to guarantee that hierarchies of 3-manifolds terminate. In the sequel \cite{TT4} we will show how powerful these locally thin positions for (3-manifold, graph) pairs are. In that paper, we construct two families of non-negative half-integer invariants of (3-manifold, graph) pairs. The invariants of one family are similar to the bridge number and tunnel number of a knot. The invariants of the other family are   very similar to Gabai's width for knots in $S^3$. We prove that these invariants (under minor hypotheses) are additive for both connect sum and trivalent vertex sum and detect the unknot. 

In Section \ref{sec:def} and Section \ref{sec:compbodies} we establish our notation and important definitions including the definition of a multiple v.p.-bridge surface. We describe our simplifying moves in Sections \ref{Reducing Complexity} and \ref{sec: elementary thinning}. In Section \ref{sec:complexity}, we define a complexity for oriented multiple v.p.-bridge surfaces and show it decreases under our simplifying moves. Section \ref{sec:complexity} also uses the simplifying moves to define a partial order $\more$ on the set $\vpoH(M,T)$ of oriented multiple v.p.-bridge surfaces for $(M,T)$. The main theorem, Theorem \ref{partial order}, shows that given $\mc{H} \in \vpoH(M,T)$ there is a least element $\mc{K} \in \vpoH(M,T)$ with respect to the partial order $\more$ such that $\mc{H} \more \mc{K}$. The least elements are called ``locally thin.'' In Section \ref{sec: sweepouts}, we study the important properties of locally thin multiple v.p.-bridge surfaces. Theorem \ref{Properties Locally Thin} lists a number of these properties, one of which is that each component of $\mc{H}^-$ is essential in the exterior of $T$.  Section \ref{thin decomp spheres} sets us up for working with connected sums in \cite{TT4} by showing that if there is a sphere in $M$, transversely intersecting $T$ in three or fewer points, and which is essential in the exterior of $T$, then there is such a sphere that is a thin level for any locally thin multiple v.p.-bridge surface. 

\subsection*{Acknowledgements} Some of this paper is similar in spirit to \cite{TT2}, but here we operate under much weaker hypotheses and obtain much stronger results. We have been heavily influenced by Gabai's work in \cite{G3}, Scharlemann and Thompson's work in \cite{ST-thin}, and Hayashi and Shimokawa's work in \cite{HS01}. Throughout we assume some familiarity with the theory of Heegaard splittings, as in \cite{Scharlemann-Survey}. We thank Ryan Blair, Marion Campisi, Jesse Johnson, Alex Zupan, and the attendees at the 2014 ``Thin Manifold'' conference for helpful conversations.  Thanks also to the referee for many helpful comments and, in particular, finding a subtle but serious error in the original version of Section \ref{sec:complexity}. The resolution of this error led to stronger results and simplified proofs.  The second named author is supported by an NSF CAREER grant and the first author by grants from the Colby College Division of Natural Sciences.

\section{Definitions and Notation} \label{sec:def}

We let $I = [-1,1] \subset \R$, $D^2$ be the closed unit disc in $\R^2$, and $B^3$ be the closed unit ball in $\R^3$. For a topological space $X$, we let $|X|$ denote the number of components of $X$. All surfaces and 3-manifolds we consider will be orientable, smooth or PL, and (most of the time) compact. If $S$ is a surface, then $\chi(S)$ is its euler characteristic. 

A \defn{(3-manifold, graph) pair} $(M,T)$ (or simply just a \defn{pair}) consists of a compact, orientable 3-manifold (possibly with boundary) $M$ and a properly embedded graph $T \subset M$. We do not require $T$ to have vertices so $T$ can be empty or a knot or link.  Since $T$ is properly embedded in $M$ all valence 1 vertices lie on $\boundary M$. We call the valence 1-vertices of $T$ the \defn{boundary vertices} or \defn{leaves} of $T$ and all other vertices the \defn{interior vertices} of $T$. We require that no vertex of $T$ have valence 0 or 2, but we allow a graph to be empty.

For any subset $X$ of $M$, let $\eta(X)$ be an open regular neighborhood of $X$ in $M$ and $\overline{\eta(X)}$ be its closure. If $S$ is a (orientable, by convention) surface properly embedded in $M$ and transverse to $T$, we write $S \subset (M,T)$. If $S \subset (M,T)$, we abuse notation slightly and write
\[
(M,T) \setminus S = (M \setminus S, T \setminus S) = (M \setminus {\eta}(S), T \setminus \eta(S)).
\]
We also write $S \setminus T$ for  $S \setminus {\eta}(T)$.  Observe that $\boundary (M \setminus T)$ is the union of $(\boundary M)\setminus T$ with $\boundary \overline{\eta(T)}$. A surface $S \subset (M,T)$ is \defn{$\boundary$-parallel} if $S \setminus T$ is isotopic relative to its boundary into $\boundary (M\setminus T)$. We say that $S \subset (M,T)$ is \defn{essential} if $S\setminus T$ is incompressible in $M\setminus T$, not $\boundary$-parallel, and not a 2-sphere bounding a 3-ball in $M\setminus T$. We say that the graph $T \subset M$ is \defn{irreducible} if whenever $S \subset (M,T)$ is a 2-sphere we have $|S \cap T| \neq  1$. The pair $(M,T)$ is \defn{irreducible} if $T$ is irreducible and if the 3-manifold $M \setminus T$ is irreducible (i.e. does not contain an essential sphere.)

We will need notation for a few especially simple (3-manifold, graph) pairs. The pair $(B^3, \text{arc})$ will refer to any pair homeomorphic to the pair $(B^3, T)$ with $T$ an arc properly isotopic into $\boundary B^3$. The pair $(S^1 \times D^2, \text{core loop})$ will refer to any pair homeomorphic to the pair $(S^1 \times D^2, T)$ where $T$ is the product of $S^1$ with the center of $D^2$. 

Finally, we will often convert vertices of $T$ into boundary components of $M$ and vice versa. More precisely, if $V$ is the union of all the interior vertices of $T$, we say that $(\mathring{M}, \mathring{T}) = (M \setminus \eta(V), T \setminus \eta(V))$ is obtained by \defn{drilling out the vertices of $T$}. Similarly, we will sometimes refer to \defn{drilling out} certain edges of $T$; i.e. removing an open regular neighborhood of those edges and incident vertices from both $M$ and $T$.

\subsection{Compressing discs of various kinds} We will be concerned with several types of discs which generalize the classical definition of a compressing disc for a surface in a 3-manifold.

\begin{definition}
Suppose that $S \subset (M,T)$ is a surface. Suppose that $D$ is an embedded disc  in $M$ such that the following hold:
\begin{enumerate}
\item $\boundary D \subset (S \setminus T)$, the interior of $D$ is disjoint from $S$, and $D$ is transverse to $T$.
\item $|D \cap T| \leq 1$
\item $D$ is not properly isotopic into $S \setminus T$ in $M \setminus T$ via an isotopy which keeps the interior of $D$ disjoint from $S$ until the final moment. Equivalently, there is no disc $E \subset S$ such that $\boundary E = \boundary D$ and $E \cup D$ bounds either a 3-ball in $M$ disjoint from $T$ or a 3-ball in $M$ whose intersection with $T$ consists entirely of a single unknotted arc with one endpoint in $E$ and one endpoint in $D$.

\end{enumerate}
Then $D$ is an \defn{sc-disc}. More specifically, if $|D \cap T| = 0$ and $\boundary D$ does not bound a disc in $S\setminus T$, then $D$ is a \defn{compressing disc}. If $|D \cap T| = 0$ and $\boundary D$ does bound a disc in $S\setminus T$, then $D$ is a \defn{semi-compressing disc}. If $|D \cap T| = 1$ and $\boundary D$ does not bound an unpunctured disc or a once-punctured disc in $S\setminus T$, then $D$ is a \defn{cut disc}. If $|D \cap T| = 1$ and $\boundary D$ does bound an unpunctured disc or a once-punctured disc in $S\setminus T$, then $D$ is a \defn{semi-cut disc}. A \defn{c-disc} is a compressing disc or cut disc. The surface $S \subset (M,T)$ is \defn{c-incompressible} if $S$ does  not have a c-disc; it is \defn{c-essential} if it is essential and c-incompressible. \end{definition}

\begin{remark}
Semi-cut discs arise naturally when $T$ has an edge containing a local knot, as in Figure \ref{Fig:semicut}. Semi-compressing discs occur in part because even though a 3-manifold $M$ may be irreducible, there is no guarantee that a given 3-dimensional submanifold is also irreducible. 
\end{remark}

\begin{center}
\begin{figure}[tbh]
\includegraphics[scale=0.5]{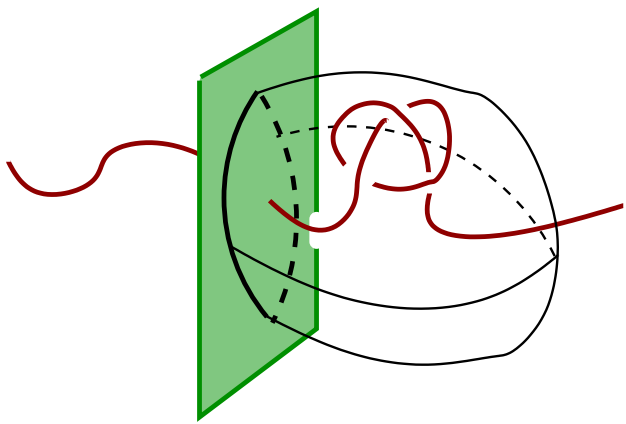}
\caption{Except in very particular situations, the black disc is a semicut disc for the green surface.}
 \label{Fig:semicut}
\end{figure}
\end{center}

\section{Compressionbodies and multiple v.p.-bridge surfaces} \label{sec:compbodies}

\subsection{Compressionbodies} 
In this section we generalize the idea of a compressionbody to our context. 
\begin{definition}
Suppose that $H$ is a closed, connected, orientable surface. We say that $(H \times I, T)$ is a \defn{trivial product compressionbody} or a \defn{product region} if $T$ is isotopic to the union of vertical arcs, and we let $\boundary_\pm (H \times I) = H \times \{\pm 1\}$. If $B$ is a 3--ball and if $T \subset B$ is a (possibly empty) connected, properly embedded, $\boundary$-parallel tree, having at most one interior vertex, then we say that $(B, T)$ is a \defn{trivial ball compressionbody}. We let $\boundary_+ B = \boundary B$ and $\boundary_- B = \nil$.  A \defn{trivial compressionbody} is either a trivial product compressionbody or a trivial ball compressionbody. Figure \ref{Fig:TrivialCompBody} shows both types of trivial compressionbodies.

\begin{center}
\begin{figure}[tbh]
\includegraphics[scale=0.5]{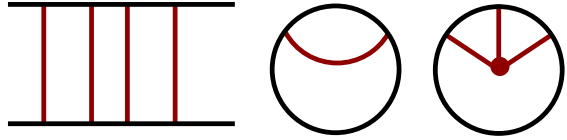}
\caption{On the left is a trivial product compressionbody; in the center is a trivial ball compressionbody with $T$ an arc; on the right is a trivial ball compressionbody with $T$ a tree having a single interior vertex..}
 \label{Fig:TrivialCompBody}
\end{figure}
\end{center}

A pair $(C,T)$ is a \defn{v.p.-compressionbody} if there is some component denoted $\bdd_+C$ of $\bdd C$ and a collection of pairwise disjoint sc-discs $\mc{D} \subset (C,T)$ for $\boundary_+ C$  such that the result of $\boundary$-reducing $(C, T)$ using $\mc{D}$ is a union of trivial compressionbodies. Observe that trivial compressionbodies are v.p.-compressionbodies as we may take $\mc{D} = \nil$. Figure \ref{Fig: vpcompressionbody} shows two different v.p.-compressionbodies. We will usually represent v.p.-compressionbodies more schematically as in Figure \ref{Fig: arc types}.

The set $\bdd C \setminus \bdd_+C$ is denoted by $\bdd_-C$. If no two discs of $\mc{D}$ are parallel in $C \setminus T$ then $\mc{D}$ is a \defn{complete collection of discs} for $(C,T)$. An edge of $T$ which is disjoint from $\boundary_+ C$ (and so has endpoints on $\boundary_- C$ and the vertices of $T$) is a \defn{ghost arc}. An edge of $T$ with one endpoint in $\boundary_+ C$ and one endpoint in $\boundary_- C$  is a \defn{vertical arc}. A component of $T$ which is an arc having both endpoints on $\boundary_+ C$  is a \defn{bridge arc}. A component of $T$ which is homeomorphic to a circle and is disjoint from $\boundary C$ is called a \defn{core loop}. $C$ is a \defn{compressionbody} if $(C,\nil)$ is a v.p.-compressionbody. A compressionbody $C$ is a \defn{handlebody} if $\boundary_- C = \nil$.  A \defn{bridge disc} for $\boundary_+ C$ in $C$ is an embedded disc in $C$ with boundary the union of two arcs $\alpha$ and $\beta$ such that $\alpha \subset \boundary_+ C$ joins distinct points of $\bdd_+C \cap T$ and $\beta$ is a bridge arc of $T$.
 \end{definition}

\begin{remark}
Suppose that $(C,T)$ is a v.p.-compressionbody and that $(\punct{C}, \punct{T})$ is the result of drilling out the vertices of $T$.  Considering the components of $\boundary \punct{C} \setminus \boundary C$ as components of $\boundary_- C$, we see that $(\punct{C}, \punct{T})$ is also a v.p.-compressionbody having the same complete collection of discs as $(C,T)$. Furthermore, every component of $\punct{T}$ is a vertical arc, ghost arc, bridge arc, or core loop. The ``v.p.'' stands for ``vertex-punctured'' as this notion of compressionbody is a generalization of the compressionbodies used in \cite{TT2}: the v.p.-compressionbody $(\punct{C}, \punct{T})$ satisfies \cite[Definition 2.1]{TT2} with $\Gamma = \punct{T}$ (in the notation of that paper.) The notation ``v.p.'' will also be helpful as a reminder that the first step in calculating many of the various quantities we consider is to drill out the vertices of $T$ and treat them as boundary components of $M$.
\end{remark}

\begin{center}
\begin{figure}[tbh]
\includegraphics[scale=0.4]{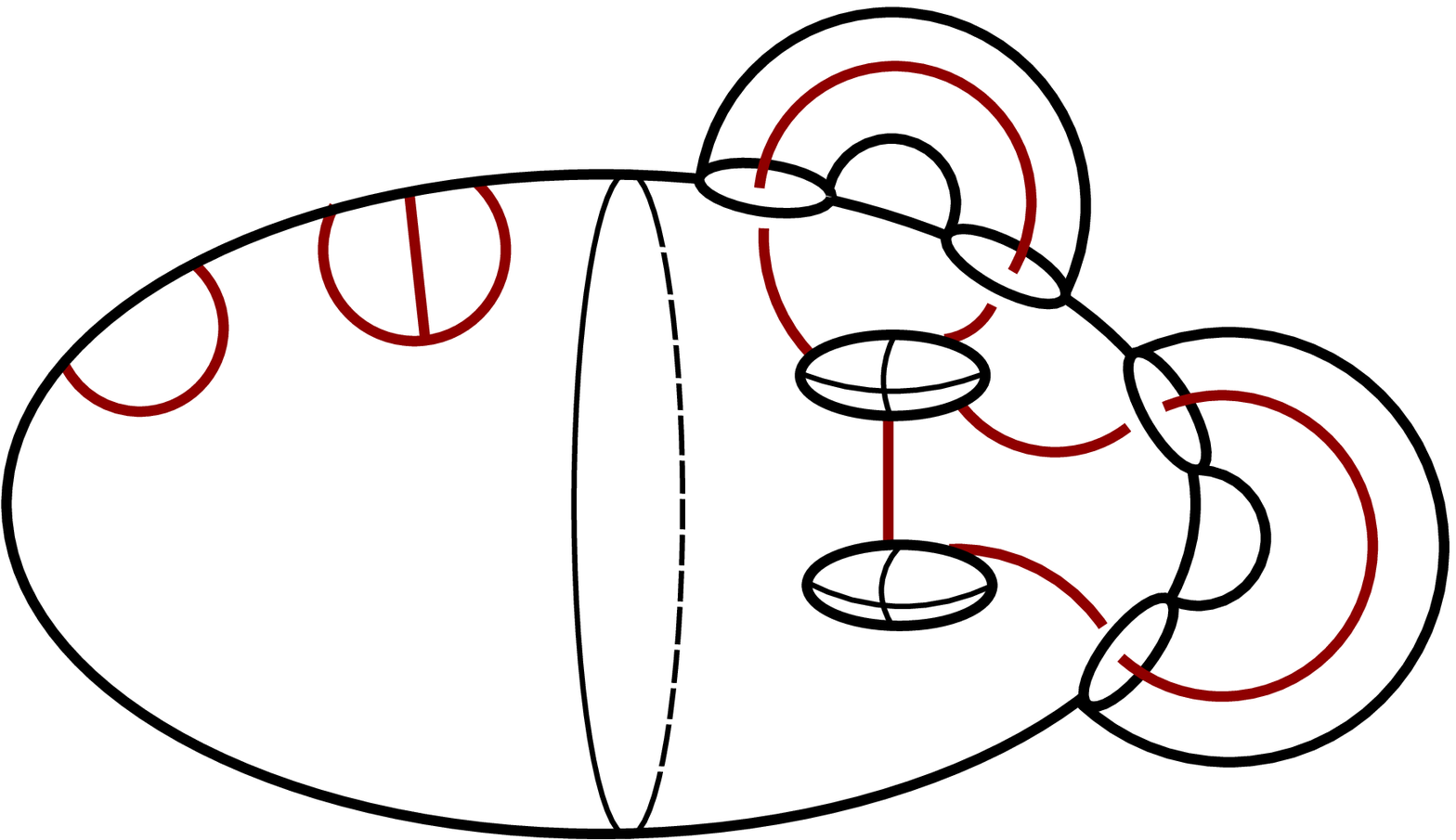}
\includegraphics[scale=0.4]{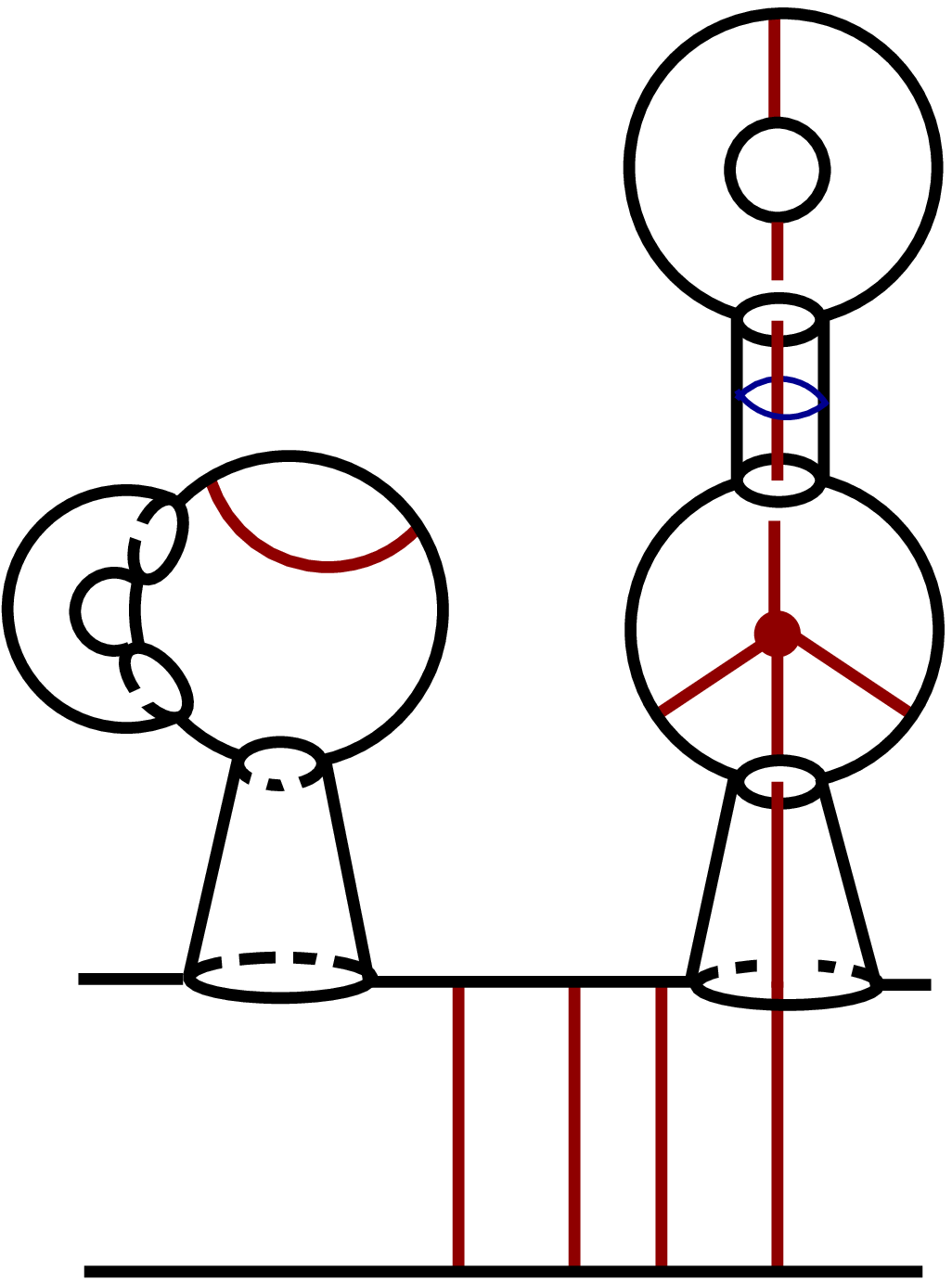}
\caption{On the left is an example of a v.p.-compressionbody $(C,T)$ with $\boundary_- C$ the union of spheres. On the right, is an example of a v.p.-compressionbody $(C,T)$ with $\boundary_- C$ the union of two connected surfaces, one of which is a sphere twice-punctured by $T$. }
 \label{Fig:  vpcompressionbody}
\end{figure}
\end{center}

\begin{center}
\begin{figure}[tbh]
\includegraphics[scale=0.4]{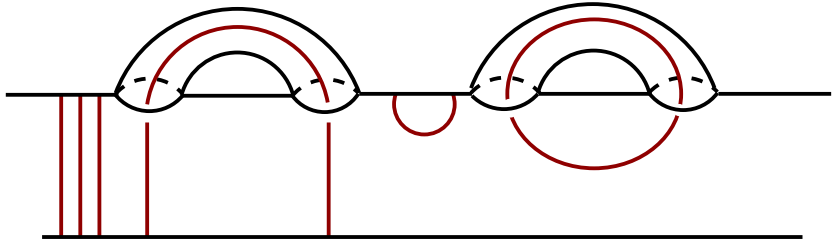}
\caption{A v.p.-compressionbody $(C,T)$. From left to right we have three vertical arcs, one ghost arc, one bridge arc, and one core loop in $T$.}
 \label{Fig:  arc types}
\end{figure}
\end{center}

The next two lemmas establish basic properties of v.p.-compressionbodies.
\begin{lemma}\label{spheres in v.p.-compressionbodies}
Suppose that $(C,T)$ is a v.p.-compressionbody such that no spherical component of $\boundary_- C$ intersects $T$ exactly once. Suppose $P \subset (C,T)$ is a closed surface transverse to $T$. If $D$ is an sc-disc for $P$, then if $|D \cap T| = 1$ either $\boundary D$ is essential on $P \setminus T$ or $\boundary D$ bounds a disc in $P$ intersecting $T$ exactly once. Furthermore, if $P$ is a sphere, then after some sc-compressions it becomes the union of spheres, each either bounding a trivial ball compressionbody in $(C,T)$ or parallel to a component of $\boundary_- C$.
\end{lemma}
\begin{proof}
Suppose first that there is an sc-disc $D$ for $P$ such that $\boundary D$ bounds an unpunctured disc $E \subset P$ but $|D \cap T| = 1$. Then $E \cup D$ is a sphere in $(C,T)$ intersecting $T$ exactly once. Every sphere in $C$ must separate $C$. Let $W \subset C \setminus (E \cup D)$ be the component disjoint from $\boundary_+ C$. The fundamental group of every component of $\boundary_- C$ injects into the fundamental group of $C$ and every curve on every component of $\boundary_- C$ is homotopic into $\boundary_+ C$. Thus, any component of $\boundary_- C$ contained in $W$ must be a sphere. Drilling out the vertices of $T$ along with all edges of $T$ disjoint from $\boundary_+ C$ creates a new v.p.-compressionbody $(C', T')$. As before, any essential curve in $\boundary_- C'$ is non-null-homotopic in $C'$ and is homotopic into $\boundary_+ C = \boundary_+ C'$. Thus, $\boundary_- C' \cap W$ can contain no essential curves. There is an edge $e \subset (T \cap W)$ with an endpoint in $D$. Beginning with $e$, traverse a path across edges of $T \cap W$ and components of $\boundary_- C \cap W$ (each necessarily a sphere) so that no edge of $T \cap W$ is traversed twice. The path terminates when it reaches a component of $\boundary_- C \cap W$ which is a once-punctured sphere, contrary to our hypotheses. Thus, no such sc-disc $D$ can exist.

Suppose now that $P$ is a sphere. Use c-discs to compress $P$ as much as possible. By the previous paragraph, we end up with the union $P'$ of spheres, each intersecting $T$ no more times than does $P$.  Let $\Delta$ be a complete collection of discs for $(C, T)$ chosen so as to minimize $|\Delta \cap P'|$ up to isotopy of $\Delta$. If some component $P_0$ of $P'$ is disjoint from $\Delta$, then it is contained in the union of trivial v.p.-compressionbody obtained by $\boundary$-reducing $(C, T)$ using $\Delta$. Standard results from 3-manifold topology show that $P_0$ is either $\boundary$-parallel to a component of $\boundary_- C$ or is the boundary of a trivial ball compressionbody in $(C, T)$. 

If $\Delta \cap P' \neq \nil$, then it consists of circles. Since we have minimized $|\Delta \cap P'|$ up  to isotopy, each circle of $\Delta \cap P'$ which is innermost on $\Delta$ bounds a semi-compressing or semi-cut disc $D \subset \Delta$ for $P'$. By the previous paragraph, compressing $P'$ using $D$ creates an additional component of $P'$ and preserves the property that each component of $P'$ intersects $T$ no more times than does $P$. The result follows by repeatedly performing such compressions until $P'$ becomes disjoint from $\Delta$. 
\end{proof}

\begin{lemma}\label{filling and puncturing}
Suppose that $(C,T)$ is a (3-manifold, graph)-pair such that no component of $\boundary_- C$ is a sphere intersecting $T$ exactly once. Then the following hold:
\begin{enumerate}
\item If $P \subset \boundary_- C$ is an unpunctured sphere or a  twice-punctured sphere, the result $(\wihat{C}, \wihat{T})$ of capping off $P$ with a trivial ball compressionbody is still a v.p.-compressionbody.
\item If $p$ is a point in the interior of $C$ (possibly in $T$), then the result of removing an open regular neighborhood of $p$ from $C$ (and $T$ if $p \in T$) is a v.p.-compressionbody.
\end{enumerate}
\end{lemma}
\begin{proof}
First, suppose that $P \subset \boundary_- C$ is a zero or twice-punctured sphere. Let $\Delta$ be a complete collection of sc-discs for $(C, T)$. Let $(C', T')$ be the result of $\boundary$-reducing $(C, T)$ using $\Delta$. Then $(C', T')$ is the union of v.p.-compressionbodies one of which is a product v.p.-compressionbody containing $P$. Capping off $P$ with a trivial ball compressionbody converts this product v.p.-compressionbody into a trivial ball compressionbody. Thus, the result of $\boundary$-reducing $(\wihat{C}, \wihat{T})$ using $\Delta$ is the union of trivial v.p.-compressionbodies. Thus, $(\wihat{C}, \wihat{T})$ is a v.p.-compressionbody.

Now suppose that $p$ is a point in the interior of $C$. Let $\Delta$ be a complete collection of sc-discs for $(C,T)$. By general position, we may isotope $\Delta$ to be disjoint from $T$. Let $(C', T')$ be the result of $\boundary$-reducing $(C,T)$ using $\Delta$. Each component of $(C', T')$ is a trivial v.p.-compressionbody, one of which $(W, T_W)$ contains $p$. If $(W, T_W)$ is a trivial ball compressionbody either with $T_W$ an arc containing $p$ or with $p$ an interior vertex of $T_W$, then the result of removing $\eta(p)$ from $(W, T_W)$ is again a trivial compressionbody, and the result follows. Suppose, therefore, that if $p \in T_W$, then either $(W, T_W) \neq (B^3, \text{ arc})$ or $p$ is not a vertex of $T_W$.

 If $p \in T$, there is a sub-arc of an edge of $T_W$ joining $\boundary_+ W$ to $p$. Let $E$ be the frontier of a regular neighborhood of that edge. Then $E$ is a semi-cut disc for $\boundary_+ W$ cutting off a v.p.-compressionbody from $(W, T_W)$ which is $(S^2 \times I, \text{two vertical arcs})$. (The fact that $E$ is a semi-cut disc follows from the considerations of the previous paragraph.) We may isotope $E$ so that $\boundary E$ is disjoint from the remnants of $\Delta$ in $\boundary_+ W$. The disc $E$ is then a cut disc or semi-cut disc for $(C, T)$ such that $\Delta \cup E$ is a collection of s.c.-discs such that $\boundary$-reducing $(C\setminus \eta(p), T\setminus \eta(p))$ is the union of trivial compressionbodies. Hence, $(C\setminus \eta(p), T\setminus \eta(p))$ is a v.p.-compressionbody. 

If $p \not\in T$, the proof is similar except we can pick any (tame) arc joining $\boundary_+ C$ to $p$ which is disjoint from $T$. 
\end{proof}

\begin{lemma}\label{Lem: Invariance}
Suppose that $(C,T)$ is a v.p.-compressionbody such that no component of $\boundary_- C$ is a 2-sphere intersecting $T$ exactly once. The following are true:
\begin{enumerate}
\item\label{it: no sc} $(C,T)$ is a trivial compressionbody if and only if there are no sc-discs for $\boundary_+ C$.
\item\label{it: no sc for neg boundary} There are no c-discs for $\boundary_- C$.
\item\label{it: reducing} If $D$ is an sc-disc for $\boundary_+ C$, then reducing $(C,T)$ using $D$ is the union of v.p.-compressionbodies. Furthermore, there is a complete collection of discs for $(C,T)$ containing $D$.  
\end{enumerate}
\end{lemma}
\begin{proof}
Proof of (1): From the definition of v.p.-compressionbody, if there is no sc-disc for $\boundary_+ C$, then $(C, T)$ is a trivial compressionbody. The converse requires a little more work, but follows easily from standard results in 3-dimensional topology.

Proof of (2): This is similar to the proof of Lemma \ref{spheres in v.p.-compressionbodies}. Suppose that $\boundary_- C$ has a c-disc $P$. As in the previous lemma, there is no sc-disc $D$ for $P$ such that $\boundary D$ bounds an unpunctured disc on $P$ but $|D \cap T| = 1$. Consequently, compressing $P$ using any sc-disc $D$ creates a new disc $P'$ intersecting $T$ no more often than did $P$ and with $\boundary P' = \boundary P$. Since $\boundary P$ is essential on $\boundary_- C \setminus T$, the disc $P'$ is a c-disc for $\boundary_- C$. Thus, we may assume that $P$ is disjoint from a complete collection of discs for $(C, T)$. It follows easily that $P$ is $\boundary$-parallel in $(C, T)$ and so is not a c-disc for $\boundary_- C$, contrary to our assumption.

Proof of (3): Let $(C,T)$ be a v.p.-compressionbody and suppose that $D$ is an sc-disc for $\boundary_+ C$.  Let $(\wihat{C}, \wihat{T})$ be the result of capping off all zero and twice-punctured sphere components of $\boundary_- C$ with trivial ball compressionbodies. By Lemma \ref{filling and puncturing}, $(\wihat{C}, \wihat{T})$ is a v.p.-compressionbody. 

If $D$ is not an sc-disc for $(\wihat{C}, \wihat{T})$, it is $\boundary$-parallel. Boundary-reducing $(\wihat{C}, \wihat{T})$ with $D$ results in two v.p.-compressionbodies: one a trivial ball compressionbody and the other equivalent to $(\wihat{C}, \wihat{T})$. Removing regular neighborhoods of certain points in the interior of $\wihat{C}$, converts $(\wihat{C}, \wihat{T})$ back into $(C, T)$. By Lemma \ref{filling and puncturing}, $\boundary$-reducing $(C,T)$ using $D$ results in v.p.-compressionbodies. A collection of sc-discs for those compressionbodies, together with $D$, gives a collection $\Delta$ of sc-discs such that $\boundary$-reducing $(C,T)$ using $\Delta$ results in trivial v.p.-compressionbodies. Thus, the lemma holds if $D$ is $\boundary$-parallel in $(\wihat{C}, \wihat{T})$. 

Now suppose that $D$ is not $\boundary$-parallel in $(\wihat{C}, \wihat{T})$. Choose a complete collection $\Delta$ of sc-discs such that $\boundary$-reducing $(\wihat{C}, \wihat{T})$ using $\Delta$ results in the union $(C', T')$ of  trivial v.p.-compressionbodies. Out of all possible choices, choose $\Delta$ to intersect $D$ minimally. We prove the lemma by induction on $|D \cap \Delta|$. 

If $|D \cap \Delta| = 0$, then $D$ is $\boundary$-parallel in $(C', T')$. In this case, the result follows easily. Suppose, therefore, that $|D \cap \Delta| \geq 1$. The intersection $D\cap \Delta$ is the union of circles and arcs. 

Suppose, first, that there is a circle of intersection. Let $\zeta \subset D \cap \Delta$ be innermost on $\Delta$. Compressing $D$ using the innermost disc $E \subset \Delta$ results in a disc $D'$ and a sphere $P$. By Lemma \ref{spheres in v.p.-compressionbodies}, if $|E \cap T| = 1$, then $|D' \cap T| = 1$ and $|P \cap T| = 2$. On the other hand, if $|E \cap T|=0$, then both $D'$ and $P$ are disjoint from $T$. By Lemma \ref{spheres in v.p.-compressionbodies}, there is a sequence of sc-compressions of $P$ which result in zero and twice-punctured spheres. (If $|E \cap T| = 0$, there are no twice-punctured spheres.) These spheres either bound trivial ball compressionbodies in $(\wihat{C}, \wihat{T})$ or are parallel to components of $\boundary_- \wihat{C}$. But since $\boundary_- \wihat{C}$ contains no zero or twice-punctured sphere components, the latter situation is impossible. The sphere $P$ is thus obtained by tubing together inessential spheres. It follows that $P$ is also inessential: it bounds a 3-ball either disjoint from $T$ or intersecting $T$ in an unknotted arc. Since $D$ is a zero or once-punctured disc, this ball gives an isotopy of $\Delta$ reducing the intersection with $D$, a contradiction.  Thus, there are no circles of intersection in $D \cap \Delta$.

Let $\zeta$ be an arc of intersection in $D \cap \Delta$ which is outermost in $\Delta$. Let $E \subset \Delta$ be the outermost disc. We may choose $\zeta$ and $E$ so that $E$ is disjoint from $T$. Boundary-reducing $D$ using $E$ results in two discs $D'$ and $D''$, at most one of which is once-punctured. Observe that a small isotopy makes both disjoint from $D$. By our inductive hypotheses applied to $D$ and $D'$, the result of $\boundary$-reducing $(C,T)$ along both of them (either individually or together) is the union $(W, T_W)$ of v.p.-compressionbodies. Remove the regular neighborhoods of points corresponding to the zero and twice-punctured sphere components of $\boundary_- C$. By Lemma \ref{filling and puncturing}, we still have v.p.-compressionbodies, which we continue to call $(W, T_W)$. Let $\mc{D}$ be the union of sc-discs for $\boundary_+ W$, including $D'$ and $D''$, such that the result of $\boundary$-reducing $(W, T_W)$ using $\mc{D}$ is the union of trivial compressionbodies. The disc $D$ is contained in one of these trivial compressionbodies and, therefore, must be $\boundary$-parallel. The result of $\boundary$-reducing $(C,T)$ along $D \cup \mc{D}$ is then the union of trivial compressionbodies, as desired.
 \end{proof}

\begin{remark} \label{rmk:semicut}
It is not necessarily the case that if $(C,T)$ is an irreducible v.p.-compressionbody then there is no sc-disc for $\boundary_- C$. To see this, let $(C,T)$ be the result of removing the interior of a regular neighborhood of a point on a vertical arc in an irreducible v.p.-compressionbody $(\tild{C},\tild{T})$. Then there is an sc-disc for $\boundary_- C$ which is boundary-parallel in $\tild{C} \setminus \tild{T}$ and which cuts off from $(C,T)$ a compressionbody which is $S^2 \times I$ intersecting $T$ in two vertical arcs. See the top diagram in Figure \ref{Fig:  CutPrimeNegBdy}. Similarly, if $\boundary_- C$ contains 2-spheres disjoint from $T$, there will be a semi-compressing disc for any component of $\boundary_- C$ which is not a 2-sphere disjoint from $T$.

We also cannot drop the hypothesis that no component of $\boundary_- C$ is a sphere intersecting $T$ exactly once. The bottom diagram in Figure \ref{Fig: CutPrimeNegBdy} shows an sc-disc with the property that boundary-reducing the v.p.-compressionbody along that disc does not result in the union of v.p.-compressionbodies.
\end{remark}

\begin{center}
\begin{figure}[tbh]
\includegraphics[scale=0.5]{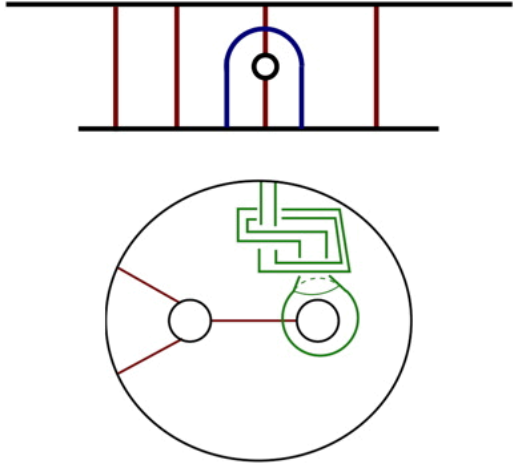}
\caption{Above is an example of a v.p.-compressionbody $(C,T)$ where $\boundary_- C$ has a semi-cut disc (shown in blue). Below is an example of an sc-disc with the property that boundary-reducing the v.p.-compressionbody along that disc does not result in the union of v.p.-compressionbodies.  }
\label{Fig:  CutPrimeNegBdy}
\end{figure}
\end{center}

In what follows, we will often use Lemma \ref{Lem: Invariance} without comment.

 \subsection{Multiple v.p.-bridge surfaces}\label{sec: multiple}

The definition of a multiple v.p.-bridge surface for the pair $(M,T)$ which we are about to present is  a version of Scharlemann and Thompson's ``generalized Heegaard splittings'' \cite{ST-thin} in the style of \cite{HS01}, but using v.p.-compressionbodies. We will also make use of orientations in a similar way to what shows up in Gabai's definition of thin position \cite{G3} and the definition of Johnson's ``complex of surfaces'' \cite{Johnson}.

 \begin{definition}\label{Def: multiple bridge surfaces}
A  connected closed surface $H \subset (M,T)$ is a  \defn{v.p.-bridge surface} for the pair $(M,T)$ if $(M,T)\setminus H$ is the union of two distinct v.p.-compressionbodies $(H_\up, T_\up)$ and $(H_\down, T_\down)$ with $H = \boundary_+ H_\up = \boundary_+ H_\down$. If $T = \nil$, then we also call $H$ a \defn{Heegaard surface} for $M$. 

A  \defn{multiple v.p.-bridge surface} for $(M,T)$ is a closed (possibly disconnected) surface $\mc{H} \subset (M,T)$ such that:
\begin{itemize}
\item $\mc{H}$ is the disjoint union of $\mc{H}^-$ and $\mc{H}^+$, each of which is the union of components of $\mc{H}$;
\item $(M,T)\setminus \mc{H}$ is the union of embedded v.p.-compressionbodies $(C_i, T_i)$ with $\mc{H}^- \cup \boundary M= \bigcup \boundary_- C_i$ and $\mc{H}^+ = \bigcup \boundary_+ C_i$;
\item Each component of $\mc{H}$ is adjacent to two distinct compressionbodies.
\end{itemize}
If $T = \nil$, then $\mc{H}$ is also called a \defn{multiple Heegaard surface} for $M$. The components of $\mc{H}^-$ are called \defn{thin surfaces} and the components of $\mc{H}^+$ are called \defn{thick surfaces}. We denote the set of multiple v.p.-bridge surfaces for $(M,T)$ by $\vpH(M,T)$. 
\end{definition}

Note that each component of $\mc{H}^+$ is a v.p.-bridge surface for the component of $(M,T)\setminus \mc{H}^-$ containing it. In particular, if $H \in \vpH(M,T)$ is connected, then it is a v.p.-bridge surface and $H^- = \nil$. Also, observe that the components of $\boundary M$ are not considered to be thin surfaces; the surfaces $\boundary M$ and $\mc{H}^-$ play different roles in what follows. We now introduce orientations and flow lines.

\begin{definition}
Suppose that $\mc{H}$ is a multiple v.p.-bridge surface for $(M,T)$. Suppose that each component of $\mc{H}$ is given a transverse orientation so that the following hold:
\begin{itemize}
\item If $(C,T_C)$ is a component of $(M,T) \setminus \mc{H}$ then the transverse orientations of the components of $\boundary_- C \cap \mc{H}^-$ either all point into or all point out of $C$.
\item If $(C,T_C)$ is a component of $(M,T) \setminus \mc{H}$, then if the transverse orientation of $\boundary_+ C$ points into (respectively, out of) $C$, then the transverse orientations of the components of $\boundary_- C \cap \mc{H}^-$ point out of (respectively, into) $C$.
\end{itemize}
A \defn{flow line} is a non-constant oriented path in $M$ transverse to $\mc{H}$, not disjoint from $\mc{H}$, and always intersecting $\mc{H}$ in the direction of the transverse orientation. If $S_1$ and $S_2$ are components of $\mc{H}$, then a flow line from $S_1$ to $S_2$ is a flow line which starts at $S_1$ and ends at $S_2$. The multiple v.p.-bridge surface $\mc{H}$ is an \defn{oriented} multiple v.p.-bridge surface if each component of $\mc{H}$ has a transverse orientation as above and there are no closed flow lines.

If there is a flow line from a thick surface $H\subset \mc{H}^+$ to a thick surface $J \subset \mc{H}^+$, then we may consider $J$ to be \defn{above} $H$ and $H$ to be \defn{below} $J$. Reversing the transverse orientation on $\mc{H}$ interchanges the notions of above and below. 
\end{definition}

The set of oriented multiple v.p.-bridge surfaces for $(M,T)$ is denoted $\vpoH(M,T)$. Note that there is a forgetful map from $\vpoH(M,T)$ to $\vpH(M,T)$. Any of our results for $\vpH(M,T)$ can be turned into results for $\vpoH(M,T)$, though the converse isn't true.

Given thick surfaces $H$ and $J$, it is not necessarily the case that $H$ is above $J$ or vice versa, even if they are in the same component of $M$. See Figure \ref{Fig:  orientedbridgesurface} for a depiction of an oriented multiple v.p.-bridge surface. Not all multiple v.p.-bridge surfaces can be oriented. For example, circular thin position (defined in \cite{Fabiola}) although we can define ``above'' and ``below'', the set of thick surfaces below a given thick surface $H$ will equal the set of thick surfaces above $H$. Notice, however, that every connected multiple v.p.-bridge surface, once it is given a transverse orientation is an oriented multiple v.p.-bridge surface since it separates $M$.
\begin{center}
\begin{figure}[tbh]
\includegraphics[scale=0.4]{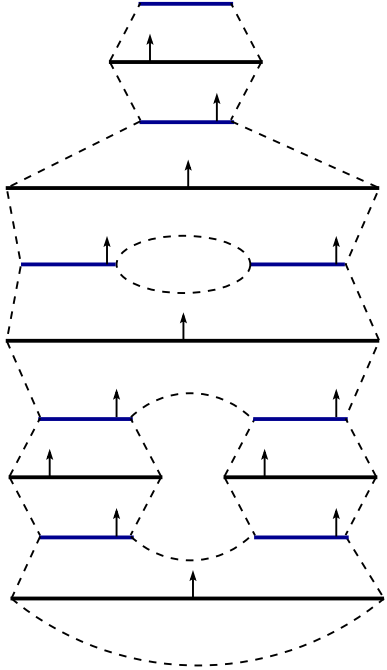}
\caption{An example of an oriented multiple v.p.-bridge surface. Blue horizontal lines represent thin surfaces or boundary components. Black horizontal lines represent thick surfaces.}
\label{Fig:  orientedbridgesurface}
\end{figure}
\end{center}

Finally, for this section, we observe that cutting open along a thin surface induces oriented multiple bridge surfaces of the components. 
\begin{lemma}\label{cutting}
Suppose that $\mc{H} \in \vpoH(M,T)$ and that $F \subset \mc{H}^-$ is a component. Let $(M', T')$ be a component of $(M,T) \setminus F$ and let $\mc{K} = (\mc{H} \setminus F) \cap M'$. Then $\mc{K} \in \vpoH(M',T')$. 
\end{lemma}

The proof of Lemma \ref{cutting} follows immediately from the definitions, as the orientation on $\mc{H}$ restricts to an orientation on $\mc{K}$ and the flow lines for $\mc{K}$ form a subset of the flow lines for $\mc{H}$.

\section{Simplifying Bridge Surfaces}\label{Reducing Complexity}

This section presents a host of ways of replacing certain types of multiple v.p.-bridge surfaces by new ones that are closely related but are ``simpler" (we will make this concept precise is section \ref{sec:complexity}). These simplifications are similar to the notion of ``destabilization'' and ``weak reduction'' for Heegaard splittings.  Versions of many of these have appeared in other papers (e.g., \cites{HS01, STo, TT1, TT2, RS}.) The operations are: (generalized) destabilization, unperturbing, undoing a removable arc, untelescoping, and consolidation.

\subsection{Destabilizing}

Given a Heegaard splitting one can always obtain a Heegaard splitting of higher genus by adding a cancelling pair of a one-handle and a two-handle, or (if the manifold has boundary) by tubing the Heegaard surface to the frontier of a collar neighborhood of a component of the boundary of the manifold. In the case where the manifold contains a graph, the core of the 1-handle, the co-core of the 2-handle, or the core of the tube might be part of the graph. (Though in this paper, we do not need to consider the case when \emph{both} the 1-handle and the 2-handle contain portions of the graph.) In the realm of Heegaard splittings, the higher genus Heegaard splitting is said to be either a stabilization or a boundary-stabilization of the lower genus one. Observe that drilling out edges of $T$ disjoint from $\mc{H}\in \vpH(M,T)$ preserves the fact that $\mc{H}$ is a multiple v.p.-bridge surface. This suggests we also need to consider boundary-stabilization along portions of the graph $T$. Without further ado, here are our versions of destabilization:

\begin{definition}\label{Def: gen stab}
Suppose that $\mc{H} \in \vpH(M,T)$ and let $H$ be a component of $ \mc{H}^+$. There are six situations in which we can replace $H$ by a new thick surface $H'$ that is obtained from $H$ by compressing along an sc-disc $D$. If $H$ satisfies any of these conditions we say that $H$ and $\mc{H}$ contain a \defn{generalized stabilization}. See Figure \ref{Fig: Stabs} for examples.
\begin{itemize}

\item There is a pair of compressing discs for $H$ which intersect transversally in a single point and are contained on opposite sides of $H$ and in the complement of all other surfaces of $\mc{H}$. In this case we say that $H$ and $\mc{H}$ are \defn{stabilized}. The pair of compressing discs is called a \defn{stabilizing pair}. The surface $H'$ is obtained from $H$ by compressing along either of the discs.

\item There is a pair of a compressing disc and a cut disc for $H$ which intersect transversally in a single point and are contained on opposite sides of $H$ and in the complement of all other surfaces of $\mc{H}$. In this case we say that $H$ and $\mc{H}$ are \defn{meridionally stabilized}. The pair of compressing disc and cut disc is called a \defn{meridional stabilizing pair}. The surface $H'$ is obtained by compressing $H$ along the cut disc.

\item There is a separating compressing disc $D$ for $H$ contained in the complement of all other surfaces of $\mc{H}$ such that the following hold. Let $W$ be the component of $M\setminus \mc{H}^-$ containing $H$. Compressing $H$ along $D$ produces two connected surfaces, $H'$ and $H''$, where $H'$ is a v.p.-bridge surface for $W$ and $H''$ bounds a trivial product compressionbody disjoint from $H'$ with a component $S$ of $\boundary M$. In this case we say that $H$ and $\mc{H}$ are \defn{boundary-stabilized} along $S$.

\item There is a separating cut disc $D$ for $H$ contained in the complement of all other surfaces of $\mc{H}$ such that the following hold. Let $W$ be the component of $M\setminus \mc{H}^-$ containing $H$. Compressing $H$ along $D$ produces two connected surfaces, $H'$ and $H''$, where $H'$ is a v.p.-bridge surface for $W$ and $H''$ bounds a trivial product compressionbody disjoint from $H'$ with a component $S$ of $\boundary M$. In this case we say that $H$ and $\mc{H}$ are \defn{meridionally boundary-stabilized} along $S$.

\item Let $G$ be a non-empty collection of vertices and edges of $T$ disjoint from $\mc{H}$. Let $\widetilde M = M \setminus G$. If $H$ and $\mc{H}$ as a multiple v.p.-bridge surface of $\widetilde M$ are (meridionally) boundary stabilized along a component of $\bdd\widetilde{M}$ which is not a component of $\bdd M$, then $H$ and $\mc{H}$ are \defn{(meridionally) ghost boundary-stabilized} along $G$.
\end{itemize}
\end{definition}

\begin{center}
\begin{figure}[tbh]
\labellist
\small\hair 2pt
\pinlabel{$H'$} [l] at 239 296
\pinlabel{$H''$} [l] at 213 243
\pinlabel{$H'$} [l] at 649 296
\pinlabel{$H''$} [l] at 622 243
\pinlabel{$H'$} [l] at 518 90
\pinlabel{$H''$} [l] at 518 55
\pinlabel{$G$} [l] at 301 26
\endlabellist
\includegraphics[scale=0.4]{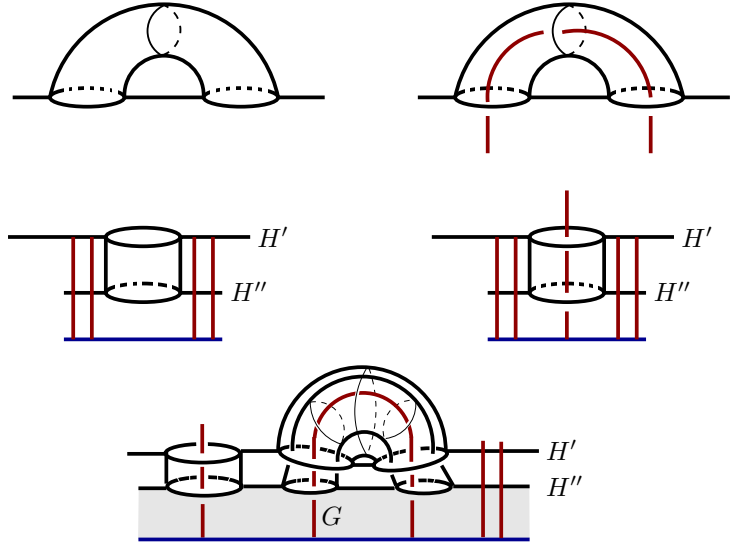}
\caption{Depictions of stabilization, meridional stabilization, $\boundary$-stabilization, meridional $\boundary$-stabilization, and meridional ghost $\boundary$-stabilization. The disc $D$ in the final picture, corresponds to the core of the left-most tube. In the last three cases, portions of the surfaces $H'$ and $H''$, which appear after compressing along the disc $D$, have been labelled. In the final case, we have shaded the product region between $H''$ and $S$.}
 \label{Fig:  Stabs}
\end{figure}
\end{center}

\begin{remark}
In the definitions of (meridional) (ghost) $\boundary$-stabilization, it's important to note that the statement that $H'$ is a v.p.-bridge surface for the component of $M \setminus \mc{H}^-$ containing it is a precondition of being able to destabilize. Not every sc-compression of a thick surface resulting in a $\boundary$-parallel surface is a destabilization. Performing a (meridional) (ghost) $\boundary$-destabilization moves one or more components of $\boundary M$ from one side of $H$ to the other side of $H'$. This is the reason we don't place transverse orientations on the components of $\boundary M$. 
\end{remark}

\begin{remark}\label{Destab Rem}
Suppose that $H \subset \mc{H}^+$ has a generalized stabilization and let $H'$ be the surface obtained from $H$ by sc-compressing as in the definition above.  It is easy to check (as in the classical settings) that $\mc{K} = (\mc{H} \setminus H) \cup H'$ is a multiple v.p.-bridge surface for $(M,T)$. If $\mc{H} \in \vpoH(M,T)$, the transverse orientation on $\mc{H}$ induces a transverse orientation on $\mc{K}$. Clearly, no new  non-constant closed flow lines are created. In particular, if $\mc{H} \in \vpoH(M,T)$, there is a natural way of thinking of $\mc{K}$ as an element of $\vpoH(M,T)$. We say that the (oriented) multiple v.p.-bridge surface $\mc{K}$ is obtained by \defn{destabilizing} $\mc{H}$ (and that the thick surface $H'$ is obtained by \defn{destabilizing} the thick surface $H$.)  
\end{remark}

\subsection{Perturbed and Removable Bridge Surfaces}
We can sometimes push a bridge surface across a bridge disc and obtain another bridge surface. This operation is called unperturbing. 

\begin{definition}
Let $\mc{H} \in \vpH(M,T)$ and let $H \subset \mc{H}^+$ be a component. Suppose that there are bridge discs $D_1$ and $D_2$ for $H$ in $M \setminus \mc{H}^-$, on opposite sides, disjoint from the vertices of $T$, and which have the property that the arcs $\alpha_1 = \boundary D_1 \cap H$ and $\alpha_2 = \boundary D_2 \cap H$ share exactly one endpoint and have disjoint interiors. Then $H$ and $\mc{H}$ have a \defn{perturbation}. The discs $D_1$ and $D_2$ are called a \defn{perturbing pair} of discs for $H$ and $\mc{H}$.
\end{definition}

\begin{remark}
The type of perturbation we have defined here might better be called an ``arc-arc''-perturbation. There are also perturbations where the bridge discs are allowed to contain vertices of $T$, but we will not need them in this paper.
\end{remark}

\begin{center}
\begin{figure}[tbh]
\includegraphics[scale=0.5]{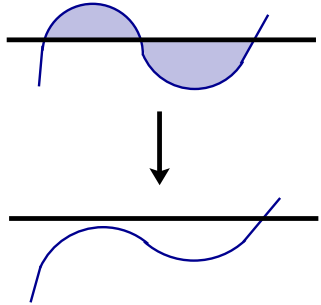}
\caption{Unperturbing $H$.}
 \label{Fig: Perturbed}
\end{figure}
\end{center}

\begin{lemma}\label{Lem: Unperturb}
Let $\mc{H}$ be an (oriented) multiple v.p.-bridge surface for $(M,T)$. Suppose that $H\subset \mc{H}^+$ is a perturbed component with perturbing discs $D_1$ and $D_2$. Let $E$ be the frontier of the neighborhood of $D_1$. Then compressing $H$ along $E$ and discarding the resulting twice punctured sphere component results in a new surface $H'$ so that $\mc{K}=(\mc{H}-H)\cup H'$ is an (oriented) multiple v.p.-bridge surface for $(M,T)$.
\end{lemma}

\begin{proof}
This is nearly identical to Lemma 3.1 of \cite{STo}. We can alternatively think of $H'$ as obtained from $H$ by an isotopy along $D_1$. On the side of $H$ containing $D_1$, this isotopy removed a bridge arc and so $H$ is still the positive boundary of a v.p.-compressionbody to that side. Let $(C, T_C)$ be the v.p.-compressionbody containing $D_2$. Let $D$ be the frontier of a regular neighborhood of $D_2$ in $C$, so that $D$ cuts off a $(B^3, \text{ arc})$ containing $D_2$ from $(C,T_C)$. Note that $D$ is an sc-disc for $(C,T_C)$. Let $\Delta$ be a complete set of sc-discs for $(C,T_C)$ containing $D$ and chosen so as to minimize $|\boundary \Delta \cap \boundary D_1|$. Observe that no component of $\Delta\setminus D$ is inside the $(B^3, \text{ arc})$ cut off by $D$. Suppose $E \subset \Delta\setminus D$ is a disc with boundary intersecting $\boundary D_1$, and which contains the intersection point of $\boundary \Delta \cap \boundary D_1$ closest to the point $\boundary D \cap \boundary D_1$. Let $E'$ be the disc obtained by tubing $E$ to a parallel copy of $D$, along a subarc of $\boundary D_1$. It is not difficult to confirm that $(\Delta \setminus E) \cup E'$ is still a complete collection of sc-discs for $(C, T_C)$. However, it intersects $\boundary D_1$ fewer times than $\Delta$, a contradiction. Thus, $\boundary D_1$ is disjoint from $\Delta \setminus D$. 

Boundary-reduce $(C, T_C)$ using $\Delta \setminus D$. We arrive at the union of v.p.-compressionbodies, one of which contains $\boundary D_1$. We can now see that the isotopy of $H$ across $D_1$, either combines two bridge arcs into another bridge arc or combines a vertical arc and a bridge arc into a bridge arc. Thus, the result of unperturbing is still a multiple v.p.-bridge surface.

If $\mc{H}$ is oriented we make $\mc{K}$ oriented by using the transverse orientations induced from $\mc{H}$. Clearly, no new closed flow lines are created.
\end{proof}

We say that the (oriented)  multiple  v.p.-bridge surface $\mc{K}$ constructed in the proof is obtained by \defn{unperturbing} $\mc{H}$. See Figure \ref{Fig: Perturbed} for a schematic depiction of the unperturbing operation.

\subsection{Removable Pairs}

Suppose that $\mc{H}$ is an (oriented) multiple v.p.-bridge surface for $(M,T)$ such that no component of $\mc{H}^- \cup \boundary M$ is a sphere intersecting $T$ exactly once. Let $H$ be a component of $\mc{H}^+$, with $\mc{D}_\up$ and $\mc{D}_\down$ complete sets of discs for $(H_\up, T_\up)$ and $(H_\down, T_\down)$ respectively. Suppose that there exists a bridge disc $D$ for $H$ in $H_\up$ (or $H_\down$) with the following properties:
\begin{itemize}
\item it is disjoint from the vertices of $T$;
\item it is disjoint from $\mc{D}_\up$ (resp. $\mc{D}_\down$);
\item  the arc $\boundary D \cap H$  intersects a single component $D^*$ of $\mc{D}_\down$ (resp. $ \mc{D}_\up$). $D^*$ is a disc and $|D \cap D^*|=1$,
\end{itemize}
then $\mc{H}$ and $H$ are \defn{removable}. The discs $D$ and $D^*$ are called a \defn{removing pair}. See the left side of Figure \ref{Fig: Removable}.

\begin{example}\label{Ex: 2-handle}
Suppose that $H \in \vpH(M,T)$ is connected and that $M'$ is obtained from $M$ by attaching a 2-handle to $\boundary M$ or Dehn-filling a torus component of $\boundary M$. Let $\alpha$ be either a co-core of the 2-handle or a core of the filling torus. Using an unknotted path in $M - H$, isotope $\alpha$ so that it intersects $H$ exactly twice. Then $H \in \vpH(M,T \cup \alpha)$ is removable. The component $\alpha$ is called the \defn{removable component of $T \cup \alpha$}.
\end{example}

\begin{lemma}\label{Removing Decr Compl.}
Suppose that $\mc{H} \in \vpH(M,T)$ is removable. Then there is an isotopy of $\mc{H}$ in $M$ to $\mc{K} \in \vpH(M,T)$  supported in the neighborhood of the removing pair so that $\mc{K}$ intersects $T$ two fewer times than $\mc{H}$ does. Furthermore, if $\mc{H}$ is oriented, so is $\mc{K}$.
\end{lemma}
\begin{proof}
Let $H$ be the thick surface which is removable. We will construct an isotopy from $H$ to a surface $H'$ supported in a regular neighborhood of the removing pair and let $\mc{K} = (\mc{H} - H) \cup H'$. We will show that $\mc{K}$ is a multiple v.p.-bridge surface. Assuming it is, if $\mc{H} \in \vpoH(M,T)$, we give $H'$ the normal orientation induced by $H$. It is then easy to show that $\mc{K} \in \vpoH(M,T)$.  

Without the loss of much generality, we may assume that $\mc{H} = H$ is connected.

Let $D \subset H_\up$ and $D^*\subset H_\down$ be the removing pair and let $\mc{D}_\up$ and $\mc{D}_\down$ be the corresponding complete set of discs from the definition of ``removable''.  Isotope $T$ across $D$ so that $T \cap D$ lies in $H_\down$. Let $T'$ be the resulting graph and let $D^*_c$ be the cut disc that $D^*$ gets converted into. Equivalently, we may isotope $H$ across $D$ and let $H'$ be the resulting surface. See Figure \ref{Fig: Removable}.

The graph $T'_\up = T' \cap H_\up$  is obtained from $T_\up = T \cap H_\up$ by removing a component of $T_\up$. After creating $T'$ from $T$, the collection $\mc{D}_\up$ remains a set of discs that decompose $(H_\up, T'_\up)$ into trivial compressionbodies, although there may now be discs in $\mc{D}_\up$ which are parallel or which are boundary-parallel in $H_\up \setminus T'_\up$. Thus, $(H_\up, T'_\up)$ is a v.p.-compressionbody.

To show that $(H_\down, T'_\down)$ is a v.p.-compressionbody, note that cut-compressing $(H_\down, T')$  along $D^*_c$ results in the same collection of compressionbodies as compressing $(H_\down, T)$ along $D^*$. Therefore $\mc{D}_\down$ with $D^*$ replaced by the induced cut disc 
$D^*_c$ is a complete collection of sc-discs for $(H_\down, T'_\down)$ and so $(H_\down, T'_\down)$ is a v.p.-compressionbody. We conclude that $\mc{K}$ is an (oriented) multiple v.p.-bridge surface. 
\end{proof}

The surface $\mc{K}$ in the preceding lemma is said to be obtained by \defn{undoing a removable arc} of $\mc{H}$.

\begin{center}
\begin{figure}[tbh]
\includegraphics[scale=0.4]{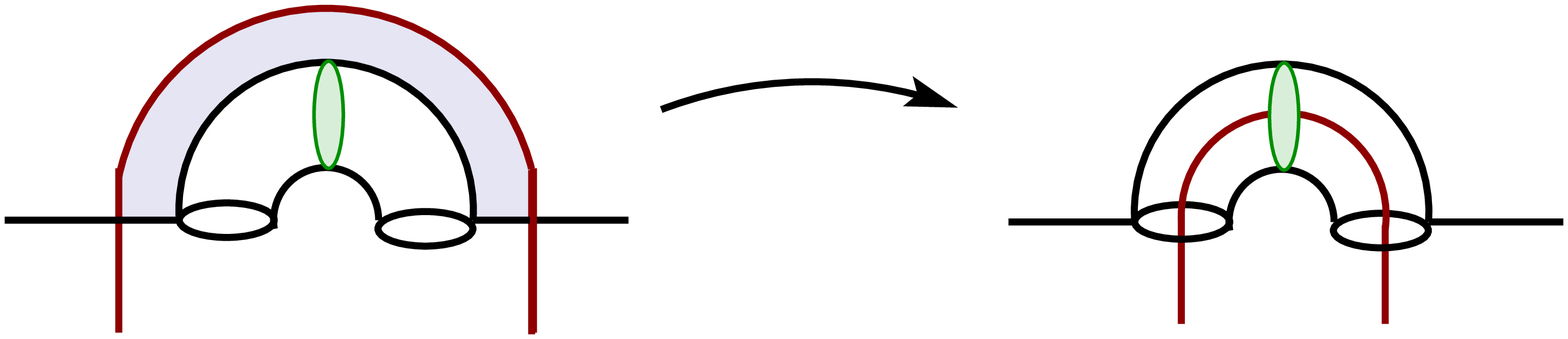}
\caption{Undoing a removable arc}
 \label{Fig: Removable}
\end{figure}
\end{center}

\section{Untelescoping and Consolidation}\label{sec: elementary thinning}

If we let $T$ be empty in everything discussed so far and if we ignore the transverse orientations, then we are in Scharlemann-Thompson's set-up for thin position. We need a way to recognize when the multiple bridge surface can be ``thinned" and a way to show that this thinning process eventually terminates. Scharlemann and Thompson thin by switching the order in which some pair (1-handle, 2-handle) are added and they use Casson-Gordon's criterion \cite{CG} to recognize that this is possible by finding disjoint compressing discs on opposite sides of a thick surface. In this section, we use compressions along sc-weak reducing pairs of discs in place of handle exchanges. 

\subsection{Untelescoping}
Suppose that $\mc{H} \in \vpH(M,T)$. If $\mc{H}$ has the property that there is a component $H \subset \mc{H}^+$, and disjoint sc-discs $D_-$ and $D_+$ for $H$ on opposite sides so that $D_-$ and $D_+$ are disjoint from $\mc{H}^-$, we say that $\mc{H}$ is \defn{sc-weakly reducible}, that $H$ is the \defn{sc-weakly reducible component} and that $\{D_-, D_+\}$ is a \defn{sc-weakly reducing pair}.  If $\mc{H}$ is not sc-weakly reducible, we say it is \defn{sc-strongly irreducible}.  If $D_-$ and $D_+$ are c-discs, we also say that $\mc{H}$ is \defn{c-weakly reducible}, etc. Suppose that no component of $\mc{H}^- \cup \boundary M$ is a sphere intersecting $T$ exactly once. Then, given an sc-weakly reducible $\mc{H} \in \vpH(M,T)$, we can create a new $\mc{K} \in \vpH(M,T)$ by \defn{untelescoping} $\mc{H}$ as follows:

\begin{definition}\label{Def: untel for unorient}
Let $\{D_-, D_+\}$ be an sc-weakly reducing pair for an sc-weakly reducible component $H$ of $\mc{H}^+$. Let $N$ be the component of $M \setminus \mc{H}^-$ containing $H$. Let $F$ be the result of compressing $H$ using both $D_-$ and $D_+$. Let $H_\pm$ be the result of compressing $H$ using only $D_\pm$ and isotope each of $H_\pm$ slightly into the compressionbody containing $D_\pm$, respectively.  Let $\mc{K}^- = \mc{H}^- \cup F$ and $\mc{K}^+ = (\mc{H}^+ \setminus H) \cup (H_- \cup H_+)$.  See Figure \ref{Fig:  untelescope} for a schematic picture. The component of $F$ adjacent to copies of both $D_-$ and $D_+$ is called the \defn{doubly spotted} component. (The terminology is taken from \cite{Scharlemann-Survey}.)
\end{definition}

\begin{center}
\begin{figure}[tbh]
\includegraphics[scale=0.4]{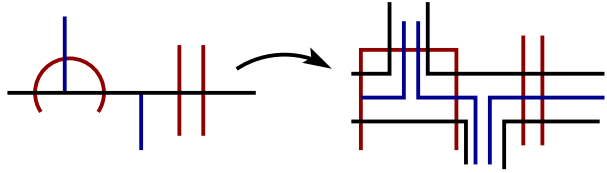}
\caption{Untelescoping $H$. The red curves are portions of $T$. The blue lines on the left are sc-discs for $H$. Note that if a semi-cut or cut disc is used then a ghost arc is created.}
 \label{Fig:  untelescope}
\end{figure}
\end{center}

\begin{lemma}\label{lem:untelescope}
If $\mc{H} \in \vpH(M,T)$ and if $\mc{K}$ is obtained by untelescoping $\mc{H}$, then $\mc{K} \in \vpH(M,T)$.
\end{lemma}
\begin{proof}
Let $H \subset \mc{H}^+$ be the component which is untelescoped using discs $\{D_-, D_+\}$. Let $W_+$ and $W_-$ be the two compressionbody components of $M\setminus \mc{H}$ that have copies of $H$ as their positive boundaries. Let $\mc{D}_\pm$ be a complete collections of discs for the compressionbodies $W_\pm$ containing $D_\pm$. The discs $\mc{D}_\pm \setminus D_\pm$ after an isotopy are a complete collection of discs for the components of $(M,T)\setminus \mc{K}$ adjacent to $H_\pm$ and not adjacent to $F$. An isotopy of the disc $D_\pm$ makes it into an sc-disc for $H_\mp$. Boundary-reducing the submanifold bounded by $H_\pm$ and $F$ using $D_\mp$ creates the union of product compressionbodies.  Thus, $\mc{K} \in \vpH(M,T)$.
\end{proof}

To extend this operation to oriented multiple v.p.-bridge surfaces we  simply give $H_-$ and $H_+$ the transverse orientations induced from $H$. We defer until Lemma \ref{Lem: niceness persists} the proof that if $\mc{H}$ is oriented, then so is $\mc{K}$. 

\subsection{Consolidation}

Untelescoping usually creates product compressionbodies which need to be removed as in Scharlemann-Thompson thin position. In our situation though this process is complicated by the presence of the graph $T$. We call the operation ``consolidation.''

\begin{definition}\label{Consolidation}
Suppose that $\mc{H}$ is an (oriented) multiple v.p.-bridge surface for $(M,T)$ and that $(P, T_P)$ is a product compressionbody component of $(M,T) \setminus \mc{H}$ which is adjacent to a component of $\mc{H}^-$ (and, therefore, not adjacent to a component of $\boundary M$.) Let $\mc{K} = \mc{H} \setminus (\boundary_- P \cup \boundary_+ P)$. If $\mc{H}$ is oriented, give each component of $\mc{K}$ the induced orientation from $\mc{H}$. We say that $\mc{K}$ is obtained from $\mc{H}$ by \defn{consolidation} or by \defn{consolidating} $(P, T_P)$. (These terms were introduced in \cite{TT2}.)
\end{definition}

The next two lemmas verify that consolidation is a valid operation in $\vpH(M,T)$.  See Figure \ref{Fig: combining indices} for a schematic depiction of the v.p.-compressionbodies in the following lemma. 

\begin{figure}[ht]
\labellist
\small\hair 2pt
\pinlabel {$A$} at 198 194
\pinlabel {$P$} at 279 98
\pinlabel {$B$} at 279 36
\pinlabel {$C$} at 652 194
\pinlabel {$\boundary_+ A$} [r] at 1 237
\pinlabel {$\boundary_- A$} [r] at 42 131
\pinlabel {$\boundary_+P = \boundary_+ B$} [r] at 213 67
\endlabellist
\includegraphics[scale=0.3]{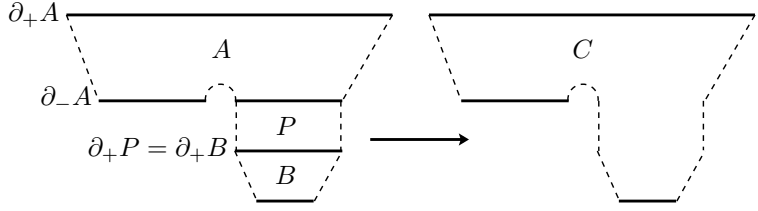}
\caption{The v.p.-compressionbodies $A$, $B$, $P$, and $C$ in Lemma \ref{Lem: Consolidation preserves v.p.-comprbdy}.}
\label{Fig: combining indices}
\end{figure}

\begin{lemma}\label{Lem: Consolidation preserves v.p.-comprbdy}
Suppose that $(P,T_P)$ is a trivial product compression body and that $(A,T_A)$ and $(B, T_B)$ are v.p.-compressionbodies with interiors disjoint from each other and from the interior of $P$.  Assume also that $\boundary _- P \subset \boundary_- A$ and $\boundary_+B = \boundary_+ P$. Let $(C,T) = (A, T_A) \cup (P, T_P) \cup (B, T_B)$ and assume that $T$ is properly embedded in $C$. Then $(C,T)$ is a v.p.-compressionbody. 
\end{lemma}

\begin{proof}
We can dually define a v.p.-compressionbody to be a 3-manifold containing a properly embedded 1-manifold obtained by taking a collection of trivial v.p.-compressionbodies and adding to their positive boundary some 1-handles and 1-handles containing a single piece of tangle as their core. With this dual definition, the lemma is obvious.
\end{proof}

\begin{lemma}\label{lem:consolidation}
Suppose that $\mc{H}$ is an (oriented) multiple v.p.-bridge surface for $(M,T)$ and that $\mc{K}$ is obtained by consolidating a product region $(P,T_P)$ of $\mc{H}$. Then $\mc{K}$ is an (oriented) multiple v.p.-bridge surface for $(M,T)$. 
\end{lemma}
\begin{proof}
This follows immediately from Lemma \ref{Lem: Consolidation preserves v.p.-comprbdy} and the observation that any closed flow line for $\mc{K}$ could be isotoped to be a closed flow line for $\mc{H}$.
\end{proof}

\subsection{Elementary Thinning Sequences}\label{sec:elem thinning sequences}

As mentioned before, untelescoping often produces product regions. These product regions, in general, are of two types -- they can be between a thin and thick surface neither of which existed before the untelescoping or they can be between a newly created thick surface and a thin surface (or a boundary component) that existed before the untelescoping operation. In fact, consolidating product regions of the first type can create additional product regions of the second type. The next definition specifies the order in which we will consolidate, before untelescoping further.

\begin{definition}\label{def:oriented elem. thinning}
Suppose that $\mc{H}$ is an sc-weakly reducible oriented multiple v.p.-bridge surface for $(M,T)$. Let $\mc{H}_1$ be obtained by untelescoping $\mc{H}$ using an sc-weak reducing pair. Let $\mc{H}_2$ be obtained by consolidating all trivial product compressionbodies of $\mc{H}_1\setminus\mc{H}$. There may now be trivial product compressionbodies in $M \setminus \mc{H}_2$. Let $\mc{H}_3$ be obtained by consolidating all those products.  We say that $\mc{H}_3$ is obtained from $\mc{H}$ by an \defn{elementary thinning sequence}. 
\end{definition}

See Figure \ref{Fig: thinning sequence} for a depiction of the creation of $\mc{H}_2$ from $\mc{H}$. 

\begin{center}
\begin{figure}[tbh]
\labellist
\small\hair 2pt
\pinlabel {$\mc{H}$} [bl] at 4 333
\pinlabel {$\mc{H}_1$} [bl] at 246 333
\pinlabel {$\mc{H}_2$} [bl] at 4 3
\pinlabel {$\mc{H}_3$} [bl] at 349 3
\endlabellist
\includegraphics[scale=0.4]{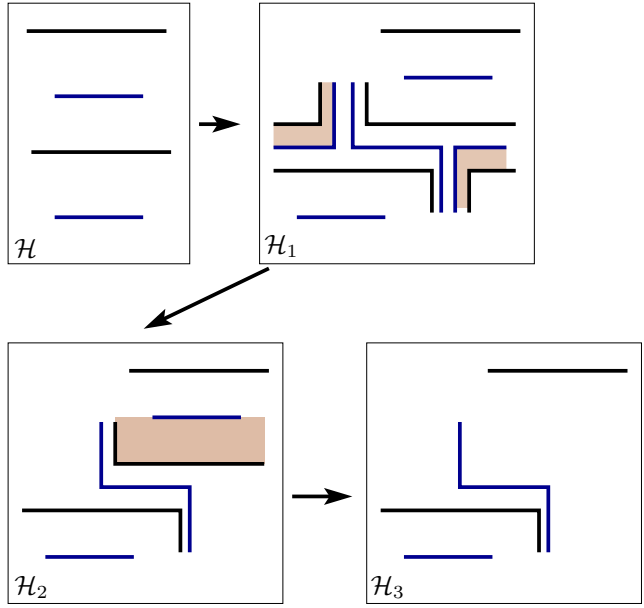}
\caption{The surface $\mc{H}_2$ is created by untelescoping and consolidation. One or both of the compressionbodies $M \setminus \mc{H}_2$ shown in the figure may be product regions adjacent to $\mc{H}^-$. We consolidate those product regions to obtain $\mc{H}_3$.}
 \label{Fig:  thinning sequence}
\end{figure}
\end{center}

To understand the effect of an elementary thinning sequence, we examine the untelescoping operation a little more carefully.

\begin{lemma}\label{lem:Sep weak reduction}
Suppose that  $H$ is a connected (oriented) v.p.-bridge surface and that $D_\up$ and $D_\down$ are an sc-weak reducing pair. Let $H_- \subset H_\down$ and $H_+ \subset H_\up$ be the new thick surfaces created by untelescoping $H$. Let $F$ be the thin surfaces. Then the following are equivalent for a component $\Phi$ of $F$:
\begin{enumerate}
\item $\Phi$ is not doubly spotted and is adjacent to a remnant of $D_\up$ (or $D_\down$, respectively).
\item The disc $D_\up$ (or $D_\down$, respectively) is separating and $\Phi$ bounds a product region in $H_\up$ (or $H_\down$, respectively) with a component of $H_+$ (or $H_-$, respectively.)
\end{enumerate}
\end{lemma}
\begin{proof}
Suppose $\Phi$ is only adjacent to $D_\up$. In this case $D_\up$ must be separating as otherwise $\Phi$ would have two spots from $D_\up$, and as $H$ is connected, it would also have to have a spot coming from $D_\down$. Compressing $H$ along $D_\up$ then results in two components. Let $H'$ be the component that doesn't contain $\bdd D_\down$. Then $H'$ is not affected by compressing along $D_\down$ to obtain $F$. Thus $H'$ is parallel to $\Phi$.

Conversely if $\Phi$ is parallel to some component $H'$ of $H_+$ say, then $\Phi$ must be disjoint from the compressing disc $D_\down$ and is therefore not double spotted. 
\end{proof}

Using the notation from Definition \ref{def:oriented elem. thinning}, we have:
\begin{lemma}\label{lem: Controlling Product Regions}
Suppose that $\mc{H} \in \vpH(M,T)$ and that $(M,T) \setminus \mc{H}$ has no trivial product compressionbodies adjacent to $\mc{H}^-$. Let $\mc{H}_1$, $\mc{H}_2$, and $\mc{H}_3$ be the surfaces in an elementary thinning sequence beginning with the untelescoping of a component $H \subset \mc{H}^+$. Then the doubly spotted component of $\mc{H}_1$ persists into $\mc{H}_3$ and no component of $(M,T) \setminus \mc{H}_3$ is a trivial product compressionbody adjacent to $\mc{H}_3^-$.
\end{lemma}
\begin{proof}
Let $H_-$ and $H_+$ be the thick surfaces resulting from untelescoping the thick surface $H \subset \mc{H}^+$ and let $F$ be the thin surface, with $F_0$ the doubly spotted component. Since $F$ is obtained by compressing using an sc-disc, $F$ is not parallel to either of $H_-$ or $H_+$. In creating $\mc{H}_2$ we remove all components of $F$ which are not doubly spotted (Lemma \ref{lem:Sep weak reduction}). The doubly spotted surface is not parallel to the remaining components of $H_-$ or $H_+$ since we can obtain it by an sc-compression of each of them. Thus, the doubly spotted component persists into $\mc{H}_2$. Let $H'_-$ and $H'_+$ be the components of $H_-$ and $H_+$ remaining in $\mc{H}_2$. If either of $H_-$ or $H_+$ bounds a trivial product compressionbody with $\mc{H}^-$, we create $\mc{H}_3$ by consolidating those trivial product compressionbodies.

Suppose that a component $(W, T_W) \subset (M, T) \setminus \mc{H}_3$ contains $F \subset \boundary_- W$.  Since $H_-$ and $H_+$ each had an sc-compression producing the doubly-spotted components of $F$, $(W, T_W)$ must contain an sc-disc for $\boundary_+ W$. Consequently, $(W, T_W)$ is not a trivial product compressionbody. The result then follows from the assumption that no component of $(M,T) \setminus \mc{H}$ was a trivial product compressionbody adjacent to $\mc{H}^-$.
\end{proof}

\begin{corollary}\label{Lem: niceness persists}
Suppose that $\mc{H}, \mc{K}$ are multiple v.p.-bridge surfaces for $(M,T)$ such that $M \setminus \mc{H}$ has no trivial product compressionbodies adjacent to $\mc{H}^-$. Assume that $\mc{K}$ is obtained from $\mc{H}$ using an elementary thinning sequence. Then the following are true:
\begin{enumerate}
\item $\mc{K}^- \neq \nil$ 
\item $\mc{K}$ has no trivial product compressionbodies disjoint from $\boundary M$, 
\item If $\mc{H}$ is oriented, so is $\mc{K}$.
\end{enumerate}
\end{corollary}
\begin{proof}
By Lemma \ref{lem: Controlling Product Regions}, the doubly spotted component of $\mc{K}^- \setminus \mc{H}^-$ does not get consolidated during the elementary thinning sequence and $\mc{K}$ has no trivial product compressionbodies adjacent to $\mc{K}^-$. 

Suppose that $\mc{H}$ is oriented. We wish to show that $\mc{K}$ is oriented. We have described how to give transverse orientations to the components of $\mc{H}_1$ and these induce transverse orientations on $\mc{H}_2$, and $\mc{K}$. It follows immediately from the construction that the transverse orientations are coherent on the v.p.-compressionbodies. We need only show that we cannot create closed flow lines. Since consolidation does not create closed flow lines, it suffices to show that $\mc{H}_1$ does not have any closed flow lines.

Suppose that $\alpha$ is a closed flow line for $\mc{H}_1$. It must  intersect $H_\pm$. As we have noted before, the (possibly disconnected) surface $H_\pm$ is obtained from $H$ by compressing along an sc-disc $D_\pm$. We can recover $H$ from $H_\pm$ by tubing (possibly along an arc component of $T\setminus \mc{H}_1$). We can isotope $\alpha$ to be disjoint from the tube, at which point it becomes a closed flow line for $\mc{H}$, a contradiction. Thus, $\mc{H}_1$, $\mc{H}_2$, and $\mc{H}_3$ are all oriented.
\end{proof}

\section{Complexity}\label{sec:complexity}

The theory of 3-manifolds is rife with various complexity functions on surfaces which guarantee certain processes terminate. In \cite {ST-thin}, Scharlemann and Thompson used a version of Euler characteristic as their measure of complexity to ensure that untelescoping (and consolidation) of Heegaard surfaces will eventually terminate. Since that foundational paper, similar complexities have been used by many authors, eg. \cites{HS01, Johnson}.  The requirement for a complexity is that it decreases under all possible types of compressions and any other moves that ``should" simplify the decomposition. In our context, we need a complexity that decreases under destabilizing a generalized stabilization, unperturbing, undoing a removable arc, and applying an elementary thinning sequence. The next example demonstrates some of the difficulties that arise in our context.

\subsection{An example}\label{Ex: Spherical Untelescoping}
Traditionally thin position in the style of Scharlemann-Thompson \cite{ST-thin} is done only for irreducible 3-manifolds. However, the following example (see Figure \ref{Fig: UnteleSpheres}) shows that, at an informal level, it should be possible to define a thin position for reducible 3-manifolds.

Let $P$ be the result of removing a regular neighborhood of two points from a 3-ball.  Choose one component of $\boundary P$ as $\boundary_+ P$ and the other two as $\boundary_- P$. Let $M$ be the result of gluing two copies of $P$ along $H = \boundary_+ P$. Then there is a certain sense in which the splitting of $M$ can be untelescoped to a simpler splitting, but the new thick surface appear more complicated. Figure \ref{Fig: UnteleSpheres} shows the original Heegaard surface and another, ostensibly thinner, multiple Heegaard surface. The surface on the right can be obtained from the one on the left by thinning using semi-compressing discs.

\begin{figure}[tbh]
\centering
\includegraphics[scale=0.4]{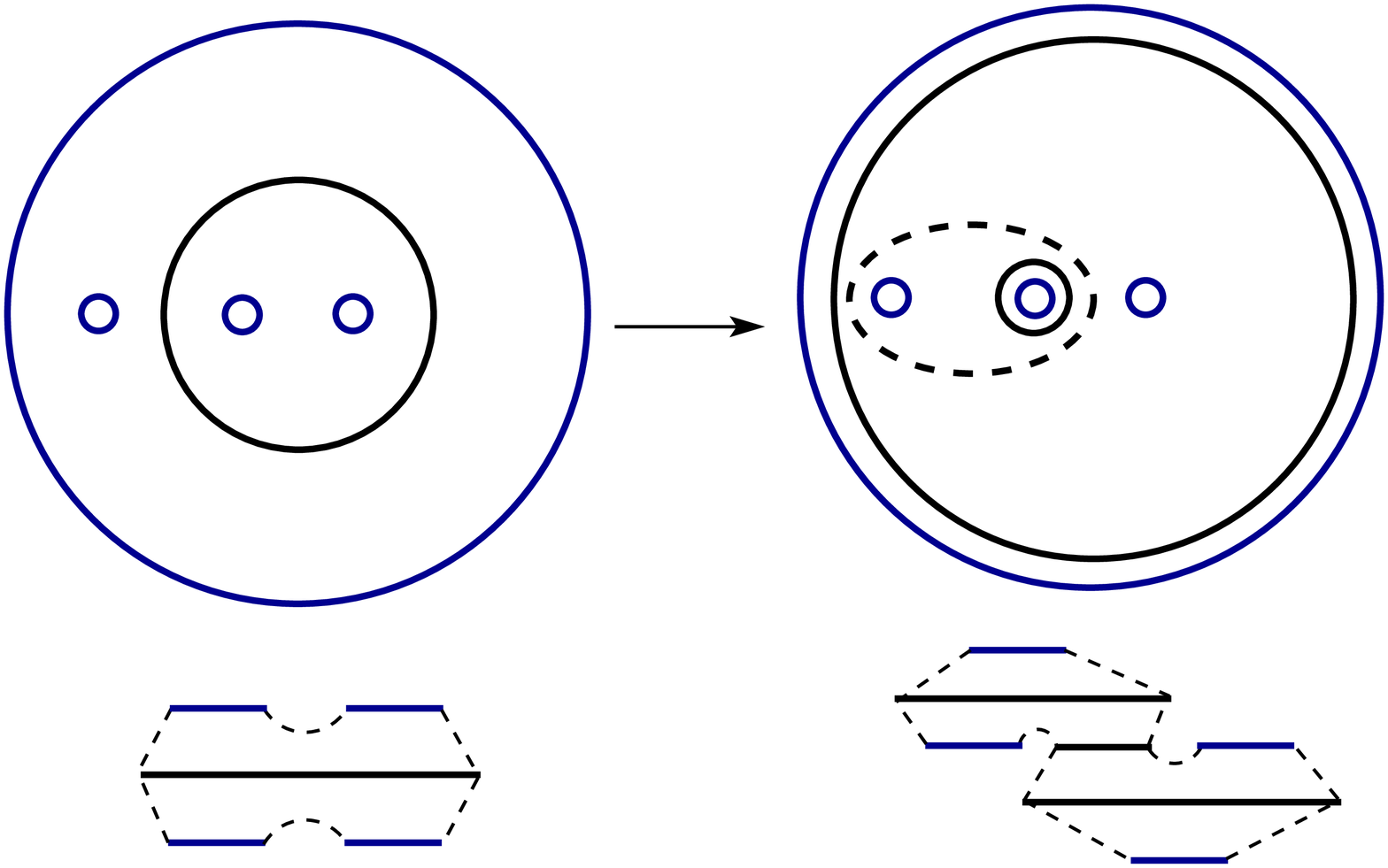}
\caption{On the left in black is the original Heegaard surface and on the right is a multiple Heegaard surface which ``should'', by all rights, be thinner. The thick surfaces are in solid black and the thin surface is in dashed black. Below each figure is a schematic representation with the boundary components in blue, the thick surfaces in long black lines, and the thin surface in a short black line.}
 \label{Fig:  UnteleSpheres}
\end{figure}

Although this example concerns a reducible manifold, we will run into similar problems when we have thin surfaces which are spheres twice-punctured by the graph $T$. If, in the example on the left, we add in a single ghost arc on each side of the Heegaard surface and four vertical arcs, one adjacent to each boundary component, we obtain an irreducible pair $(M,T)$ with a connected v.p.-bridge surface that can be thinned to the surface on the right using semi-cut discs. Observe that neither of the v.p.-compressionbodies in the example on the left contains a compressing disc or a cut disc and that neither is a trivial v.p.-compressionbody.

\subsection{Index of v.p.-compressionbodies}

We introduce the index of a v.p.-compressionbody (see below) as a first step is developing a useful complexity for oriented multiple v.p.-bridge surfaces. This index is a proxy for counting handles. The index of a compressionbody without an embedded graph was first defined by Scharlemann and Schultens \cite{Scharlemann-Schultens-JSJ}. 

\begin{definition}
For a v.p.-compressionbody  $(C, T_C)$ such that $T_C$ does not have interior vertices, define
\[
\mu(C, T_C) = 3(-\chi(\boundary_+ C) + \chi(\boundary_- C)) + 2(|\boundary_+ C \cap T| - |\boundary_- C \cap T|)+6.
\]
If $T_C$ does have interior vertices, drill them out and then calculate $\mu$. For convenience, define $\mu(\nil) = 0$.
\end{definition}

\begin{remark}
The +6 isn't strictly needed, but allows us to work with non-negative integers.
\end{remark}

Observe that $\mu(B^3, \nil) = 0$; $\mu(B^3, \text{ arc}) = 4$; and the index of any other trivial v.p.-compressionbody is 6. Since the euler characteristic of a closed surface is even, index is always even. The next lemma is proved by considering the effect of a $\boundary$-reduction on $\mu$. 

\begin{lemma}\label{lem:compressing}
Suppose that $(C, T_C)$ is a v.p.-compressionbody such that no component of $\boundary_- C$ is a sphere intersecting $T_C$ exactly once. If $D \subset (C, T_C)$ is an sc-disc for $\boundary_+ C$ and if $(C_1, T_1)$ and $(C_2, T_2)$ are the result of $\boundary$-reducing $(C, T_C)$ using $D$ (we allow $(C_2, T_2)$ to be empty) then
\begin{equation}\label{eq: boundary red}
\mu(C_1, T_1)+ \mu(C_2, T_2)= \mu(C, T_C) - 6 + 4|D \cap T_C|+6\delta
\end{equation}
where $\delta=1$ if $D$ is separating and 0 otherwise. Consequently, $\mu(C_1, T_1) < \mu(C, T)$. 

Furthermore, for any v.p.compressionbody $\mu(C, T_C) \geq 0$ with $\mu(C, T_C) = 0$ if and only if $(C, T_C)=(B^3, \nil)$; $\mu(C, T_C)=4$ if and only if $(C, T_C)=(B^3, arc)$; and in all other cases $\mu(C, T_C)\geq 6$.
\end{lemma}

\begin{proof}
If $T_C$ has interior vertices, drill them out. Suppose first that $D \subset (C, T_C)$ is an sc-disc for $\boundary_+ C$. Recall that $|D \cap T_C| \in \{0,1\}$. Let $\Delta$ be a complete collection of sc-discs containing $D$. Let $(C', T') = (C_1, T_1) \cup (C_2, T_2)$. We prove the lemma by induction on $|\Delta|$. 

Let $(C', T')$ be the result of $\boundary$-reducing $(C, T_C)$ using $D$. Considering the effect of $\boundary$-reduction on euler characteristic and the number of punctures produces Equation \eqref{eq: boundary red}. Notice that $\Delta \setminus D$ is a complete collection of sc-discs for $(C', T')$.  

If $|\Delta| = 1$, then $(C', T')$ is the union of trivial v.p.-compressionbodies. If $\delta = 0$, then $(C, T_C)$ is either $(S^1 \times D^2, \nil)$ or $(S^1 \times D^2, \text{ core loop})$. In either case, $\mu(C, T_C) = 6$ and $\mu(C_1, T_1) = \mu(C', T')$ is either 0 or 4. Suppose $\delta = 1$. Since $D$ is an sc-disc, neither $(C_1, T_1)$ nor $(C_2, T_2)$ is $(B^3, \nil)$. Similarly, if $|D \cap T_C| = 1$, then neither can be $(B^3, \text{ arc})$. In particular, 
\[
\mu(C_1, T_1) = \mu(C, T_C) + 4|D \cap T_C| - \mu(C_2, T_2) < \mu(C, T_C).
\]

The proof of the inductive step is similar; we apply the inductive hypothesis to $(C_2, T_2)$ to conclude that $\mu(C_2, T_2) \geq 6$ when it is non-trivial.
\end{proof}

The next lemma considers the effect of consolidation on index. See Figure \ref{Fig: combining indices} for a diagram.

\begin{lemma}\label{lem:trivialconsolidation}
Suppose that $\mc{H}\in \vpH(M,T)$ is a multiple v.p.-compressionbody. Suppose $(A, T_A)$, $(P, T_P)$ and $(B, T_B)$ are v.p.-compressionbodies with $(P, T_P)$ a product v.p.-compressionbody such that $\bdd_-P  \subset \bdd_- A$ and $\bdd_+B = \bdd_+P$. Let $C=A \cup P \cup B$ and $T = T_A \cup T_P \cup T_B$. Then $\mu(C, T_C)=\mu(A, T_A)+\mu(B, T_B)-6$.
\end{lemma}

\begin{proof}
By the definition of product v.p.-compressionbody, $-\chi(\boundary_+ P) = -\chi(\boundary_- P)$ and $|\boundary_+ P \cap T| = |\boundary_- P \cap T|$. Let $\alpha = \boundary_- A \setminus \boundary_- P$. Recall that $\boundary_+ C = \boundary_+ A$ and $\boundary_- C = \alpha \cup \boundary_- B$. We have:
\[\begin{array}{rcl}
\mu(A, T_A) + \mu(B, T_B) &=& 3(-\chi(\boundary_+ A) + \chi(\alpha)) + 2(|\boundary_+ A \cap T| - |\alpha \cap T|)  + 6 \\
&& + 3\chi(\boundary_- B) - 2|\boundary_- B \cap T| + 6 \\
&& + 3\chi(\boundary_- P) - 2|\boundary_- P \cap T| - 3\chi(\boundary_+ B) + 2|\boundary_+ B \cap T|\\
&=& \mu(C, T) + 6.
\end{array}
\]
\end{proof}

For a thick surface $H \subset \mc{H}^+$, let $\mu_\downarrow(H)=\mu(H_\downarrow)$ and $\mu_\uparrow(H)=\mu(H_\uparrow)$.  We now define the oriented indices $I_\up(\mc{H})$ and $I_\down(\mc{H})$. These will contribute to a complexity which decreases under all relevant moves. Informally, for each thick surface we calculate the sum of the number of ``handles'' which are immediately above some thick surface which is either $H$ or above $H$ and the number of ``handles'' which are immediately below some thick surface which is either equal to $H$ or below $H$. We place these numbers into a non-increasing sequence and compare the results lexicographically. Instead of working with ``handles'', however, we use the indices of v.p.-compressionbodies.

\begin{definition}
Let $\mc{H} \in \vpoH(M,T)$. Each v.p.-compressionbody $(C, T_C) \subset (M,T) \setminus \mc{H}$ is adjacent to a single thick surface $H \subset \mc{H}^+$. The transverse orientation on $H$ either points into or out of $C$. If it points into $C$, then $(C, T_C) = H_\up$ and if it points out of $C$, then $(C, T_C) = H_\down$. In the former case we say that $(C, T_C)$ is an \defn{upper} v.p.-compressionbody for $\mc{H}$ and we say it is a \defn{lower} v.p.-compressionbody in the latter case. 

Consider the set of flow lines beginning at $H$. A v.p.-compressionbody component (other than $H_\down$) of $(M,T)\setminus \mc{H}$ intersecting one of these flow lines is said to be \defn{above} $H$. We say that a v.p.-compressionbody is \defn{below} $H$ if reversing the transverse orientation of $\mc{H}$ makes it above $H$. Define ${\mc{H}^H_\up}$ to be the set of all upper compression bodies $J_\up$ above $H$. Define ${\mc{H}^H_\down}$ to be the set of all lower compression bodies $J_\down$ below $H$. Since there are no closed flow lines, the sets $\mc{H}^H_\up$ and $\mc{H}^H_\down$ are disjoint.

Define the \defn{upper index} and \defn{lower index} of $H$ to be (respectively):
\[\begin{array}{rcl}
 I_{\up}(H)&=& 6 - 6|\mc{H}^H_\up| + \sum\limits_{J_\up \in \mc{H}^H_\up} \mu_\up(J) \\
I_{\down}(H)&=& 6 - 6| \mc{H}^H_\down|+ \sum\limits_{K_\down \in \mc{H}^H_\down}\mu_\down(K).
\end{array}
\]
\end{definition}

In Lemma \ref{lem:bounded below} below, we verify that both $I_\up(H)$ and $I_\down(H)$ are non-negative.

To package the indices for thick surfaces into an invariant for $\mc{H}$, let $\oc(\mc{H})$, the \defn{oriented complexity} of $\mc{H}$, be the \emph{non-increasing} sequence whose terms are the quantities $I(H)=I_\up(H) + I_\down(H)$ for each thick surface $H \subset \mc{H}^+$. 

\subsubsection{Oriented complexity decreases under generalized destabilization, unperturbing, and undoing a removable arc}
\begin{lemma}\label{lem:I decreases under gen destab}
Assume that no component of $\boundary M$ is a sphere intersecting $T$ in two or fewer points. Suppose that $\mc{K}$ is obtained from $\mc{H}$ by a generalized destabilization, unperturbing, or undoing a removable arc. Then $\oc(\mc{K}) < \oc(\mc{H})$.
\end{lemma}
\begin{proof}
Let $H \subset \mc{H}^+$ be the thick surface to which we apply the generalized destabilization, unperturbing, or undoing a removable arc. Let $H'$ be the new thick surface, so that $\mc{K} = (\mc{H} \setminus H) \cup H'$. 

Suppose first that we are performing a destabilization or meridional destabilization. In this case, $-\chi(H') = -\chi(H) - 2$ and $|H'\cap T|$ is either $|H \cap T|$ or $|H \cap T| + 2$. Thus, $\mu_\down(H') < \mu_\down(H)$ and $\mu_\up(H') < \mu_\up(H)$. It follows that $I(H') < I(H)$ and for every thick surface $J \subset \mc{H}^+\setminus H$, the index $I(J)$ does not increase under the destabilization or meridional destabilization. 

The cases when we unperturb or undo a removable arc are very similar: we simply use the fact that $|H' \cap T| = |H \cap T| - 2$. 

Now suppose that we perform a $\boundary$-destabilization, meridional $\boundary$-destabilization, ghost $\boundary$-destabilization, or meridional ghost $\boundary$-destabilization. Since indices are calculated by drilling out the interior vertices of $T$, we may assume that there are none.  In all these cases, there is a (possibly disconnected) closed subsurface  $S \subset \boundary M$ and a (possibly empty) subset $\Gamma$ of edges of $T$ each disjoint from $\mc{H}$. These are such that $H'$ is the result of compressing $H$ along a separating sc-disc $D$ and then discarding a component $H''$ which is the frontier of the regular neighborhood of $S \cup \Gamma$. In particular, the euler characteristic of the discarded component is $\chi(S) - 2|\Gamma|$.

We claim that $\mu_\down(H') < \mu_\down(H)$ and $\mu_\up(H') < \mu_\up(H)$. Let $p = |D \cap T|$. Observe that $|H' \cap T| = |H \cap T| - |S \cap T| + p$. (If $p = 1$, this follows from the fact that $D$ is separating.) Also note that $-\chi(H') = -\chi(H) + \chi(S) - 2|\Gamma| - 2$. Hence,
\[
-3\chi(H') + 2|H' \cap T| = -3\chi(H)  + 2|H \cap T|  + 3 \chi(S) - 6|\Gamma| - 2|S \cap T| + 2p - 6.
\]
Without loss of generality, we may assume that $S \cup \Gamma \subset H_\up$. The (ghost) (meridional) $\boundary$-stabilization then moves $S \cup \Gamma$ to the lower compressionbody $H'_\down$. That is, $S \cup \Gamma \subset H'_\down$. In particular, $\boundary_- H'_\up = \boundary_- H_\up \setminus S$ and $\boundary_- H'_\down = \boundary_- H_\down \cup S$. Thus, we have:
\[
\begin{array}{rcl}
\mu_\up(H') &=& \mu_\up(H) + (3\chi(S) - 6|\Gamma| - 2|S \cap T| + 2p - 6) + (- 3\chi(S) + 2|S \cap T|) \\
\mu_\down(H') &=& \mu_\down(H) + (3\chi(S) - 6|\Gamma| - 2|S \cap T| + 2p - 6) +  (3\chi(S) - 2|S \cap T|)\\
\end{array}
\]
The first term in parentheses in each equation comes from the change of $H$ to $H'$ and the second term comes from the movement of $S \cup \Gamma$ from $\boundary_- H_\up$ to $H'_\down$.

Simplifying, and using the fact that $2p \in \{0,2\}$, we obtain:
\[
\begin{array}{rcl}
\mu_\up(H') &\leq & \mu_\up(H)  - 6|\Gamma| - 4  \\
\mu_\down(H') &\leq& \mu_\down(H) - 6|\Gamma| + 6 \chi(S) - 4|S \cap T|  - 4 \\
\end{array}
\]
In particular, $\mu_\up(H') < \mu_\up(H)$. The situation for $\mu_\down$ requires more analysis. Let $S_0 \subset S$ be the subset which is the union of all spherical components of $S$ and let $S_1 = S\setminus S_0$. We have
\[
\mu_\down(H') \leq \mu_\down(H) - 6|\Gamma| + 12|S_0| - 4|S_0 \cap T|  - 4 \\
\]
By assumption, each component of $S_0$ intersects $T$ at least three times, so $4|S_0 \cap T| \geq 12|S_0|$. Thus, $\mu_\down(H') < \mu_\down(H)$.

Since $S \subset \boundary M$ and does not belong to $\mc{H}'$, we can conclude that for each thick surface $J \subset \mc{H}\setminus H$, the indices $I_\up(J)$ and $I_\down(J)$ do not increase under the (meridional) (ghost) $\boundary$-destabilization. Furthermore, since $I_\up(H') + I_\down(H') < I_\up(H) + I_\down(H)$, we have
\[
\oc(\mc{K}) < \oc(\mc{H}),
\]
as desired.
\end{proof}

\subsubsection{Oriented complexity decreases under consolidation}
\begin{lemma}\label{lem:consolidation leaves I}
Suppose that $\mc{K} \in \vpoH(M,T)$ is obtained from $\mc{H}\in \vpoH(M,T)$ by consolidating a thick surface $H \subset \mc{H}^+$  with a thin surface $Q \subset \mc{H}^-$. Then $\oc(\mc{K}) < \oc(\mc{H})$.
\end{lemma}
\begin{proof}
Without loss of generality, we may suppose that $Q \subset \boundary_- (H_\down)$. (If not, reverse orientations so that above and below are interchanged.) That is, $H_\down$ is the product compressionbody bounded by $H$ and $Q$. Let $C \neq H_\down$ be the other v.p.-compressionbody such that $Q \subset \boundary_- C$. Let $J \subset \mc{H}^+\setminus H$ be another thick surface. We will show that $I_\up(J)$ and $I_\down(J)$ are unchanged by the consolidation.

The v.p.-compressionbodies of $(M,T) \setminus \mc{K}$ are obtained from those of $(M,T) \setminus \mc{H}$ by replacing $C$, $H_\down$, and $H_\up$ with their union. By Lemma \ref{lem:trivialconsolidation}, we have
\begin{equation}\label{muconsolthin}
\mu(C \cup H_\down \cup H_\up) = \mu(C) + \mu(H_\up) - 6.
\end{equation}
The consolidation does not affect flow lines, and so if there is no flow line from $J$ to $H$ or from $H$ to $J$, then $I_\up(J)$ and $I_\down(J)$ are clearly unaffected. 

If there is a flow line from $J$ to $H$, then Equation \eqref{muconsolthin} implies that $I_\up(J)$ decreases by 6. But we also have $|\mc{K}^J_\up| = |\mc{H}^J_\up| - 1$, and so $I_\up(J)$ also increases by 6. Thus, $I_\up(J)$ is unchanged by the consolidation. Clearly, $I_\down(J)$ is also unchanged by the consolidation, because the consolidation happens above $J$.

If there is a flow line from $H$ to $J$, then clearly $I_\up(J)$ is unchanged by the consolidation. On the other hand, $I_\down(J)$ decreases by 6 because we have removed $\mu(H_\down)$ from the sum. However, $I_\down(J)$ also increases by 6 since $|\mc{K}^J_\down| = |\mc{H}^J_\down| - 1$. Thus, $I_\down(J)$ is also unchanged by the consolidation.

Thus, $\oc(\mc{K})$ is simply obtained from $\oc(\mc{H})$ by removing the term $I_\up(H) + I_\down(H)$. Since the sequence was non-increasing, we have $\oc(\mc{K}) < \oc(\mc{H})$.
\end{proof}

\subsubsection{Oriented complexity decreases under an elementary thinning sequence}

Suppose that $\mc{H}$, $\mc{H}_1$, $\mc{H}_2$, $\mc{H}_3 = \mc{K}$ are the multiple v.p.-bridge surfaces in an elementary thinning sequence obtained by untelescoping a thick surface $H \subset \mc{H}^+$ using sc-discs $D_-$ and $D_+$. As we've done before, let $H_-$ and $H_+$ be the new thick surfaces and $F$ the new thin surface. We will generally work with $\mc{H}_2$ and $\mc{H}_3$ (rather than $\mc{H}_1$) so $H_\pm$ is obtained from $H$ by compressing along $D_\pm$ and discarding a component if $\delta_\pm = 1$. The surface $F$ is then obtained from $H_\pm$ by compressing along $D_\mp$ and possibly discarding a component.

\begin{lemma} \label{lem:mu goes down}
The following hold for $H_-, H_+ \subset \mc{H}^+_2$.
\begin{enumerate}
\item $\mu_\downarrow(H_-) < \mu_\downarrow(H)$
\item $\mu_\uparrow(H_+)< \mu_\uparrow(H)$
\item $\mu_\downarrow(H_-) + \mu_\downarrow(H_+) = \mu_\downarrow(H)+6$
\item $\mu_\uparrow(H_-)+\mu_\uparrow(H_+) =  \mu_\uparrow(H)+6$.
\end{enumerate}
\end{lemma}

\begin{proof}
Claims (1) and (2) follow immediately from Lemma \ref{lem:compressing}. Claim (4) can be obtained from the proof of Claim (3) by interchanging $+$ and $-$ and $\up$ and $\down$. We prove Claim (3).

Let $p_- = |D_- \cap T|$. Let $\delta_- = 1$, if $\boundary D_-$ separates $H$ and 0 otherwise.  If $\delta_- = 1$, let $R_\down$ be the v.p.-compressionbody such that $(H_-)_\down \cup R_\down$ is the result of $\boundary$-reducing $H_\down$ using $D_-$. See Figure \ref{fig:Rregion}.  If $\delta_- = 0$, then let $R_\up = \nil$ and recall that $\mu(R_\up) = 0$, by convention. 

\begin{figure}[ht]
\labellist
\small\hair 2pt
\pinlabel {$D_+$} [bl] at 52 111
\pinlabel {$D_-$} [tr] at 154 40
\pinlabel {$(H_-)_\down$} at 429 28
\pinlabel {$R_\down$} at 618 28
\pinlabel {$(H_+)_\down$} at 618 96
\endlabellist
\includegraphics[scale = 0.5]{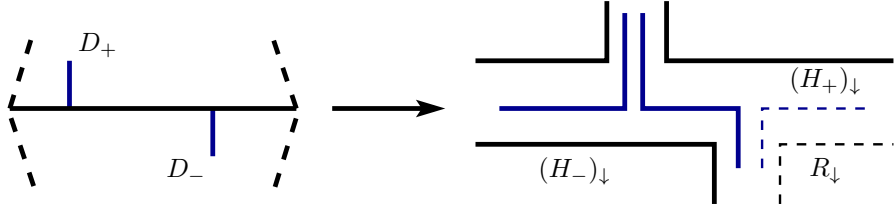}
\caption{The region between the thin surface and thick surface, both indicated with dashed lines, is consolidated in the passage from $\mc{H}_1$ to $\mc{H}_2$ when $\boundary D_-$ separates $H$. The v.p.-compressionbody $R_\down$ is then a subset of $(H_+)_\down$. }
\label{fig:Rregion}
\end{figure}

From the definition of index (see Lemma \ref{lem:compressing}) we have
\[
\mu_\down(H_-) + \mu(R_\down) = \mu_\down(H) - 6 + 4p_-  + 6\delta_-.
\]

If $\delta_- = 0$, notice that $\mu_\down(H_+) = 12 - 4p_-$, since a single compression creates $F$ from $H_+$. If $\delta_- = 1$, then before the consolidation that creates $\mc{H}_2$ from $\mc{H}_1$, the index of $(H_+)_\down$ is again $12 - 4p_-$. The consolidation removes a surface parallel to $\boundary_+ R_\up$ from the negative boundary of the v.p.-compressionbody and replaces it with $\boundary_- R_\up$. Recalling the additional (+6) term in the definition of $\mu$ we have, in either case,
\[
\mu_\down(H_+) = 12 - 4p_- + \mu(R_\down) - 6\delta_-
\]

Thus,
\[
\mu_\down(H_-) + \mu_\down(H_+) = \mu_\down(H) + 6.
\]
\end{proof}

\begin{corollary}\label{corollary:I decreases}
The following hold for $H_-, H_+ \subset \mc{H}^+_2$:
\begin{enumerate}
\item $I_\down(H_-)<I_\down(H)$
\item $I_\up(H_-) =  I_\up(H)$
\item $I_\down(H_+) = I_\down(H)$
\item $I_\up(H_+) <  I_\up(H)$
\end{enumerate}
\end{corollary}

\begin{proof}
We prove (2) and (4). Conclusions (1) and (3) follow by reversing the orientation on $\mc{H}$.

Observe that each flow line beginning at $H$ extends to a flow line beginning at $H_-$ and can be restricted to a flow line beginning at $H_+$. Thus, the set $(\mc{H}_2)^{H_-}_\up$ is obtained from the set $\mc{H}^H_\up$ by removing the v.p.-compressionbody $H_\up$ and replacing it with the v.p.-compressionbodies $(H_-)_\up$ and $(H_+)_\up$. Observe that $|(\mc{H}_2)^{H_-}_\up| = |\mc{H}^H_\up| + 1$. Hence, using Lemma \ref{lem:mu goes down} part (4),
\[\begin{array}{rcl}
I_\up(H_-) &=& I_\up(H) - \mu_\up(H) + \mu_\up(H_-) + \mu_\up(H_+) - 6 \\
&=& I_\up(H) + 6  - 6 \\
&=& I_\up(H).
\end{array}.
\]
This proves Conclusion (2).

On the other hand, $(\mc{H}_2)^{H_+}_\up$ is obtained from $\mc{H}^H_\up$ by removing $H_\up$ and replacing it with $(H_+)_\up$. From Lemma \ref{lem:mu goes down}, part (3) we have $\mu_\up(H_+) < \mu_\up(H)$. Hence, $I_\up(H_+) < I_\up(H)$, proving Conclusion (4).
\end{proof}

\begin{corollary}\label{cor:elem thinning decreases I}
If $\mc{K}$ is obtained from $\mc{H}$ by an elementary thinning sequence, then $\oc(\mc{K}) < \oc(\mc{H})$.
\end{corollary}
\begin{proof}
Let $\mc{H}_1$, $\mc{H}_2$, and $\mc{H}_3 = \mc{K}$ be the multiple v.p.-bridge surfaces created during the elementary thinning sequence. Corollary \ref{corollary:I decreases}, shows that $I(H_\pm) < I(H)$. 

If $J \subset \mc{H}^+ \setminus H$, then both $I_\up(J)$ and $I_\down(J)$ are unchanged in passing from $\mc{H}$ to $\mc{H}_2$. To see this, consider the possible locations of $J$. If $J$ is neither above nor below $H$, then $J$ is neither above nor below either of  $H_-$ nor $H_+$ and so $I(J)$ is unchanged. If $J$ is below $H$, then $J$ is below both $H_-$ and $H_+$, as any flow line from $J$ to $H$ extends to a flow line from $J$ to $H_+$ passing through $H_-$. In this case, passing from $\mc{H}$ to $\mc{H}_2$ increases the number of v.p.-compressionbodies above $J$ by 1 and does not change the number of v.p.-compressionbodies below $J$. In the calculation of $I_\up(J)$, we replace $\mu_\up(H)$ with $\mu_\up(H_-) + \mu_\up(H_+)$. By Lemma \ref{lem:mu goes down}, this increases the sum of the indices of the upper v.p.-compressionbodies above $J$ by 6. It does not change the sum of the indices of the lower v.p.-compressionbodies below $J$. Thus, $I(J)$ does not increase when passing from $\mc{H}$ to $\mc{H}_2$. The analysis when $J$ is above $H$ is nearly identical.

We may conclude, therefore, that $\oc(\mc{H}_2) < \oc(\mc{H})$. Finally, either $\mc{H}_3 = \mc{H}_2$ or $\mc{H}_3$ is obtained from $\mc{H}_2$ by one or two consolidations. Thus, by Lemma \ref{lem:consolidation leaves I}:
\[
\oc(\mc{H}_3) \leq \oc(\mc{H}_2) < \oc(\mc{H}),
\]
as desired.
\end{proof}

\subsection{Index is non-negative}

The next lemma will help ensure that our oriented complexity guarantees that we cannot perform an infinite sequence of simplifying moves on an oriented v.p.-compressionbody.

\begin{lemma}\label{lem:bounded below}
Suppose that no component of $\mc{H}^-$ is a sphere intersecting $T$ exactly once. Then for any  thick surface $H \subset \mc{H}^+$, both $I_\up(H)$ and $I_\down(H)$ are non-negative.
\end{lemma}

\begin{proof}
We prove the statement for $I_\up(H)$; the proof of the statement for $I_\down(H)$ is nearly identical. We may assume that $T$ has no interior vertices (drill them out if necessary). We will also work under the assumption that no v.p.-compressionbody of $(M,T)\setminus \mc{H}$ is a product adjacent to a component of $\mc{H}^-$. To see that we may do this, recall from the proof of Lemma \ref{lem:consolidation leaves I} that consolidation leaves $I_\up(H)$ unchanged if $H$ is not consolidated. If $H$ is consolidated, and $H_\down$ is the product region, then it is easy to see that either $I_\up(H) = \mu(H_\up) \geq 0$ or $I_\up(H)$ is equal to $\mu(H_\up) + I_\up(J) \geq I_\up(J)$ for some thick surface $J$ above $H$. Finally, if $H_\up$ is the product region, then there exists a thick surface $J$ above $H$ such that $\boundary_- H_\up \subset \boundary_- J_\down$. We may calculate $I_\up(H)$ from $I_\up(J)$ by subtracting 6 since $\mc{H}^H_\up = \mc{H}^J_\up \cup H_\up$ and also adding 6 since $\mu(H_\up) = 6$. Thus, if $I_\up(J)$ is non-negative, so is $I_\up(H)$. Henceforth, we assume that $(M,T)\setminus \mc{H}$ has no product regions adjacent to a component of $\mc{H}^-$.

We can express the definition of $I_\up(H)$ as:
\[
 I_{\up}(H) = 6 + \sum_{J_\up \in \mc{H}^H_\up} (\mu(J_{\up})- 6) \\
\]

By Lemma \ref{lem:compressing}, $\mu(J_{\up}) \geq 6$ unless $J_{\up }$ is $(B^3, \emptyset)$ or $(B^3, \text{ arc})$ in which cases $\mu(J_{\up })=0$ and  $\mu(J_{\up}) =4$ respectively. Thus, if no element of $\mc{H}^H_\up$ is a trivial ball compressionbody, then  $I_\up(H) \geq 0$. Assume, therefore that at least one element of $\mc{H}^H_\up$ is a trivial ball compressionbody. We induct on the number $N(H, \mc{H})$ of trivial ball compressionbodies in $\mc{H}^H_\up$.

If $H_\up$ is $(B^3, \emptyset)$ or $(B^3, \text{ arc})$, then $|\mc{H}^H_\up| = 1$ and $I_\up(H) = \mu(H_\up) \in \{0,4\}$, as desired. We may assume, therefore, that $|\mc{H}^H_\up| \geq 2$.

\begin{figure}[ht]
\labellist
\small\hair 2pt
\pinlabel{$C_1$} at 142 195
\pinlabel{$C_2$} at 237 73
\pinlabel{$V$} at 147 147
\endlabellist
\includegraphics[scale=0.5]{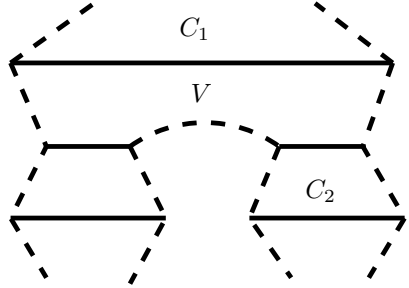}
\caption{The v.p.-compressionbodies $C_1$ and $C_2$ are adjacent in $\mc{H}^H_\up$ for $H = \boundary_+ C_2$ or any thick surface $H$ below $\boundary_+ C_2$. In this example, there is another possible choice for $C_2$.}
\label{fig:upperadjacent}
\end{figure}

We will call  v.p.-compressionbodies $C_1, C_2 \in \mc{H}^H_\up$ \defn{adjacent in $\mc{H}^H_\up$} if there is a v.p.-compressionbody $V$ such that $\boundary_+C_1=\boundary_+V$ and $\boundary_-C_2\cap \boundary_-V \neq \emptyset$ or vice versa. See Figure \ref{fig:upperadjacent} for an example. If $C_1 \in \mc{H}^H_\up$ is a trivial ball compressionbody, then there must be a v.p.-compressionbody $C_2 \in \mc{H}^H_\up$ adjacent in $\mc{H}^H_\up$ to $C_1$ as there is a flow line from $H$ to $\boundary_+ C_1 \neq H$. Observe that in such a situation, $C_2$ is not a trivial ball compressionbody (since $\boundary_- C_2 \neq \nil$).

Furthermore, if $C_1$ is a trivial ball compressionbody adjacent in $\mc{H}^H_\up$ to $C_2$, with $V$ the lower compressionbody incident to both, then $\boundary_+ V$ is a zero or twice-punctured sphere. Consequently $\boundary_- V$ is the union of spheres. Let $\Gamma$ be the graph with vertices the components of $\boundary_- V$ and edges corresponding to the ghost arcs in $V$. Since $\boundary_+ V$ is a sphere, $\Gamma$ is the union of isolated vertices and trees. Since no component of $\boundary_- V$ is a once-punctured sphere, each component $P$ of $\boundary_- V$ is a zero or twice-punctured sphere. In particular, if $C_1 \cap T = \nil$, then $P$ is unpunctured.

Suppose now that $A \in \mc{H}^H_\up$ is adjacent in $\mc{H}^H_\up$ to a trivial ball compressionbody $C \in \mc{H}^H_\up$. Choose a single component $P$ of $\boundary_- A$ such that a flow line from $H$ to $\boundary_+ C$ passes through $P$. This implies that $P \subset \boundary_- V$ where $V$ is the lower v.p.-compressionbody incident to both $A$ and $C$. By the remarks of the previous paragraph (with $C_1 = C$ and $C_2 = A$), $P$ is a zero or twice punctured sphere (as are all components of $\boundary_- V$).

Cut $(M, T)$ open along all components of $\boundary_- V \setminus P $, turning those components into components of $\boundary M$ which are zero or twice-punctured spheres. Let $\mc{H}'$ be the components of $\mc{H}$ which are not now components of $\boundary M$. Observe that $\mc{H}'^+ = \mc{H}^+$ and that now every flow line from $H$ to $\boundary_+ C$  must pass through $P$. We have not, however, changed $I_\up(H)$ since any compressionbody of $(M,T)\setminus \mc{H}$ which was above $H$ is still a compressionbody of $(M,T)\setminus \mc{H}'$ above $H$ and we have not created any new v.p.-compressionbodies above $H$. We may have disconnected $M$; however, any v.p.-compressionbodies not in the component of $M$ containing $H$ were not above $H$ before the cut and we can ignore them for the purposes of the calculation. For convenience of notation, use $\mc{H}$ instead of $\mc{H}'$ and assume that every flow line from $H$ to $\boundary_+ C$ must pass through $P$. 

Now cut open $M$ along $P$. This cuts $M$ into two components $M_1$ and $M_2$ with $M_1$ containing $H$ and $M_2$ containing $C$. Cap off the components of $\boundary M_1$ and $\boundary M_2$ corresponding to $P$ with $(B^3, \nil)$ or $(B^3, \text{ arc})$ corresponding to whether or not $P$ is a zero or twice punctured sphere. Let $(\wihat{M}_i, \wihat{T}_i)$ for $i = 1,2$ be these new (3-manifold, graph) pairs. Observe that $\wihat{\mc{H}} = \mc{H}\setminus (P \cup \boundary_+ C)$ is a multiple v.p.-bridge surface for $(\wihat{M}_1, \wihat{T}_1)$. 

The only upper v.p.-compressionbody affected by this is $A$; we obtain a new v.p.-compressionbody $\wihat{A}$. If $\wihat{A}$ is a trivial ball compressionbody, then $P = \boundary_- A$; $A$ contains no bridge arcs; and $\boundary_+ A$ is a sphere. This is enough to guarantee that $A$ is a product compressionbody, contrary to hypothesis. Thus, with respect to $\wihat{\mc{H}}$ there is one fewer trivial ball compressionbody above $H$ than with respect to $\mc{H}$. Let $\wihat{I}$ be $I_\up(H)$ with respect to $\wihat{\mc{H}}$ and let $I$ be $I_\up(H)$ with respect to $\mc{H}$. By our inductive hypothesis, we have $\wihat{I} \geq 0$. 

We have
\[
\mu(A) = \mu(\wihat{A}) + 6 - 2|P \cap T|.
\]

Thus,
\[
I = \wihat{I} + 6 - 2|P \cap T| + (\mu(C) - 6) \geq \mu(C) - 2|P \cap T|
\]
Recalling that $\mu(C) \in \{0,4\}$ and $|P \cap T| \in \{0,2\}$, we need only realize that if $\mu(C) = 0$, then $|P \cap T| = 0$ to conclude that
\[
I \geq 0.
\]
\end{proof}

\begin{remark}\label{rmk:terminates} By Lemma \ref{lem:bounded below}, each term of $\oc(\mc{H})$ is non-negative. Thus, any sequence of multiple v.p.-bridge surfaces $\mc{H}$ with $\oc(\mc{H})$ strictly decreasing must terminate.
\end{remark}

\subsection{Extended thinning moves}

In this section, we formalize the fact that oriented complexity forbids an infinite sequence of simplifying moves to an oriented multiple v.p.-compressionbody.

\begin{definition}
An oriented multiple v.p.-bridge surface $\mc{H}$ is \defn{reduced} if it does not contain a generalized stabilization, a perturbation, or a removable arc and if no component of $(M,T)\setminus \mc{H}$ is a trivial product compressionbody adjacent to a component of $\mc{H}^-$.
\end{definition}

\begin{definition}
Suppose that $\mc{H}\in \vpoH(M,T)$ is reduced and that $T$ is irreducible.  An \defn{extended thinning move} applied to $\mc{H}$ consists of the following steps in the following order:
\begin{enumerate}
\item Perform an elementary thinning sequence
\item Destabilize, unperturb, and undo removable arcs until no generalized stabilizations, perturbations, or removable arcs remain
\item Consolidate all components of $\mc{H}^-$ and $\mc{H}^+$ cobounding a trivial product compressionbody in $(M,T) \setminus \mc{H}$
\item Repeat (2) and (3) as much as necessary until $\mc{H}$ does not have a generalized stabilization, perturbation, or removable arc or product region adjacent to $\mc{H}^-$.
\end{enumerate}
\end{definition}

\begin{remark}\label{Step 2 before Step 3}
Corollary \ref{cor:elem thinning decreases I},  Lemma \ref{lem:I decreases under gen destab}, and Lemma \ref{lem:consolidation leaves I} show that each of the steps (1), (2), (3), if applied non-vacuously, strictly decrease oriented complexity. Thus, by Remark \ref{rmk:terminates} they can occur only finitely many times, until either we cannot (non-vacuously) perform any of the steps of an extended thinning move or until we have a multiple v.p.-bridge surface having a thin level which is a sphere intersecting $T$ exactly once.  

We have phrased the steps as we have in order to guarantee that if $\mc{H}$ is reduced, then an extended thinning move applied to $\mc{H}$ results in a reduced multiple v.p.-bridge surface. If $\mc{H} \in \vpoH(M,T)$ is not reduced, we may perform a sequence of consolidations, generalized destabilizations, unperturbings, and undoings of removable arcs to make it reduced. (Such a sequence is guaranteed to terminate because each of those operations strictly decreases oriented complexity.)
\end{remark}

\begin{definition}
If $\mc{H}, \mc{K} \in \vpoH(M,T)$ then we write $\mc{H} \more \mc{K}$ if either of the following holds:
\begin{itemize}
\item $\mc{H}$ is reduced and $\mc{K}$ is obtained from $\mc{H}$ by an extended thinning move, or
\item $\mc{H}$ is not reduced, $\mc{K}$ is reduced and $\mc{K}$ is obtained from $\mc{H}$ by a sequence of consolidations, generalized destabilizations, unperturbings, and undoing of removable arcs.
\end{itemize}
We then extend the definition of $\more$ so that it is a partial order on $\vpoH(M,T)$. In particular, if $\mc{H}$ is reduced, then $\mc{H} \more \mc{K}$ means that $\mc{K}$ is obtained from $\mc{H}$ by a (possibly empty) sequence of extended thinning moves. 
\end{definition}

Recall that in a poset, a ``least element'' is an element  $x$ with the property that no element is strictly less than $x$. In our context, we say that an element $\mc{K} \in \vpoH(M,T)$ is a \defn{least element} or \defn{locally thin} if it is reduced and if $\mc{K} \more \mc{K}'$ implies that $\mc{K} = \mc{K}'$. 

The following result follows immediately from our work above. The hypothesis that $T$ is irreducible guarantees that in a sequence of extended thinning moves we never have a thin surface which is a sphere intersecting $T$ exactly once.

\begin{theorem}\label{partial order}
Let $(M,T)$ be a (3-manifold, graph) pair with $T$ irreducible. Suppose that no component of $\boundary M$ is a sphere intersecting $T$ two or fewer times. Then, for all $\mc{H} \in \vpoH(M, T)$ there is a least element (i.e. locally thin) $\mc{K} \in \vpoH(M,T)$ such that $\mc{H} \more \mc{K}$.
\end{theorem}

\section{Sweepouts}\label{sec: sweepouts}

Sweepouts, as in most applications of thin position, are the key tool for finding disjoint compressing discs on two sides of a thick surface. In this section, we will use $X - Y$ to denote the set-theoretic complement of $Y$ in $X$, as opposed to $X\setminus Y$ which indicates the complement of an open regular neighborhood of $Y$ in $X$.

\begin{definition}
Suppose that $(C,T)$ is a v.p.-compressionbody and that $\Sigma \subset C$ is a trivalent graph embedded in $C$ such that the following hold:
\begin{itemize}
\item $(C,T)\setminus\Sigma$ is homeomorphic to $(\boundary_+ C \times I, \text{vertical arcs})$ 
\item $\Sigma$ contains the ghost arcs of $T$ and no interior vertex of $\Sigma$ lies on a ghost arc
\item Each boundary vertex of $\Sigma$ lies on $T$ or on $\boundary_- C$.
\item Any edge of $T$ which is not a ghost arc and which intersects $\Sigma$ is a bridge arc intersecting $\Sigma$ in a boundary vertex.
\end{itemize}
 Then $\Sigma$ is a \defn{spine} for $(C,T)$. See Figure \ref{Fig: vpSpine} for an example.
\end{definition}

\begin{center}
\begin{figure}[tbh]
\includegraphics[scale=0.4]{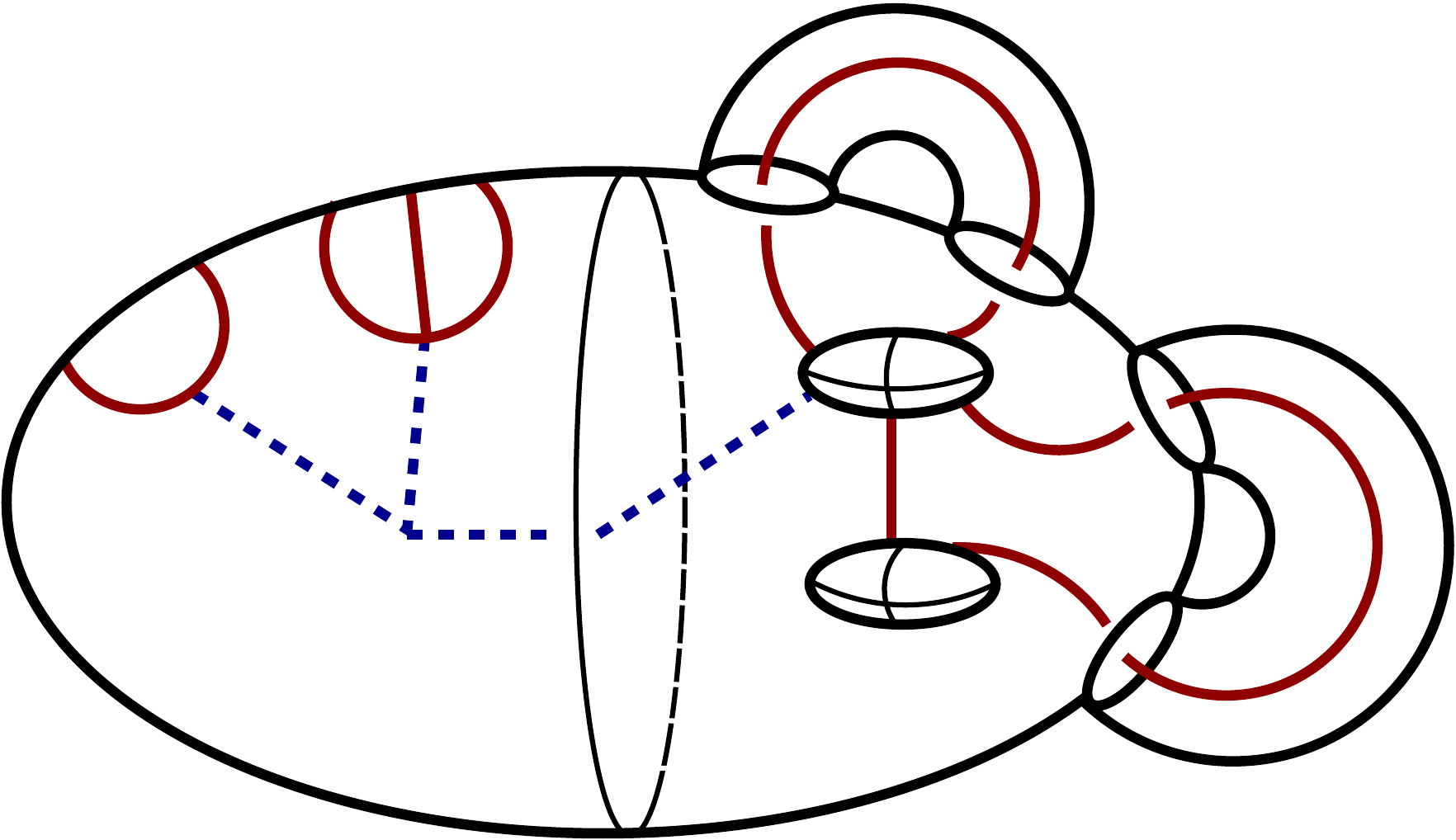}
\caption{A spine for the v.p.-compressionbody from Figure \ref{Fig:  vpcompressionbody} consists of the dashed blue graph together with the edges of $T$ that are disjoint from $\boundary_+ C$.}
\label{Fig:  vpSpine}
\end{figure}
\end{center}

Suppose that $H \in \vpH(M,T)$ is connected and that $\Sigma_\up$ and $\Sigma_\down$ are spines for $(H_\up, T \cap H_\up)$ and $(H_\down, T \cap H_\down)$. The manifold $M - (\Sigma_\up \cup \Sigma_\down)$ is homeomorphic to $H \times (0,1)$ by a map taking $T - (\Sigma_\up \cup \Sigma_\down)$ to vertical edges. We may extend the homeomorphism to a map $h\co M \to I$ taking $\Sigma_\down$ to $-1$ and $\Sigma_\up$ to $+1$. The map $h$ is called a \defn{sweepout} of $M$ by $H$. Note that for each $t \in (0,1)$, $H_t = h^{-1}(t)$ is properly isotopic in $M \setminus T$ to $H \setminus T$, that $h^{-1}(-1) = \boundary_- H_\down \cup \Sigma_\down$, and $h^{-1}(1) = \boundary_- H_\up \cup \Sigma_\up$. If we perturb $h$ by a small isotopy, we also refer to the resulting map as a sweepout.

\begin{theorem}\label{Thm: Sweepout}
Let $(M,T)$ be a (3-manifold, graph) pair. Suppose that $F \subset (M,T)$ is an embedded surface and assume that $H \in \vpH(M,T)$ is connected and doesn't bound a trivial v.p.-compressionbody on either side. Then, $H$ can be isotoped transversally to $T$ such that after the isotopy  $H$ and $F$ are transverse and one of the following holds
 \begin{enumerate}
\item\label{it: disjoint} $H \cap F = \nil$
\item\label{it: essential} $H \cap F \neq \nil$, every component of $H \cap F$ is essential in $F$ and no component of $H \cap F$ bounds an sc-disc for $H$.
\item\label{it: sc-weakly red} $H$ is sc-weakly reducible.
\end{enumerate}
\end{theorem}

\begin{remark}
The essence of this argument can be found in many places. It originates with Gabai's original thin position argument \cite{G3}  and is adapted to the context of Heegaard splittings by Rubinstein and Scharlemann \cite{RubSch}. A version for graphs in $S^3$ plays a central role in \cite{GST}. 
\end{remark}

\begin{proof}
Let $h$ be a sweepout corresponding to $H$, as above. Perturb the map $h$ slightly so that $h|_F$ is Morse with critical points at distinct heights. Let 
\[
0 =v_0 < v_1 < v_2 < \cdots < v_n = 1
\]
be the critical values of $h|_F$. Let $I_i = (v_{i-1}, v_i)$. Label $I_i$ with $\down$ (resp. $\up$) if some component of $F \cap H_t$ bounds a sc-disc below (resp. above) $H_t$ for some $t \in I_i$.

Observe that, by standard Morse theory, the label(s) on $I_i$ are independent of the choice of $t \in I_i$.

\textbf{Case 1:} Some interval $I_i$ is without a label.

Let $t \in I_i$. If $H_t \cap F = \nil$, then we are done, so suppose that $H_t \cap F \neq \nil$.

Suppose that some component $\zeta \subset H_t \cap  F$ is inessential in $F$. This means that $\zeta$ bounds an unpunctured or once-punctured disc in $F$. Without loss of generality, we may assume that $\zeta$ is innermost in $F$. Let $D \subset F$ be the disc or once-punctured disc it bounds. Since $\zeta$ does not bound an sc-disc for $H_t$, the disc $D$ is properly isotopic in $M\setminus T$, relative to $\boundary D$ into $H_t$. Let $B \subset M$ be the 3-ball bounded by $D$ and the disc in $H_t$. By an isotopy supported in a regular neighborhood of $B$, we may isotope $H_t$ to eliminate $\zeta$ (and possibly some other inessential curves of $H_t \cap F$.). Repeating this type of isotopy as many times as necessary, we may assume that no curve of $H_t \cap F$ is inessential in $F$. If $H_t \cap F = \nil$, we have the first conclusion. If $H_t \cap F \neq \nil$, then we have the Conclusion \eqref{it: essential}. 

Suppose, therefore, that each $I_i$ has a label.

\textbf{Case 2:} Some $I_i$ is labelled both $\up$ and $\down$.

Since for each $t \in I_i$, $H_t$ is transverse to $F$ we have Conclusion \eqref{it: sc-weakly red}.

\textbf{Case 3:} There is an $i$ so that $I_i$ is labelled $\down$ and $I_{i+1}$ is labelled $\up$, or vice versa.

The labels cannot change from $I_i$ to $I_{i+1}$ at any tangency other than a saddle tangency. Let $\epsilon > 0$ be smaller than the lengths of the intervals $I_i$ and $I_{i+1}$. Since $H_t$ is orientable, under the projections of $H_{v_i - \epsilon}$ and $H_{v_i + \epsilon}$ to $H$, the  1-manifold  $H_{v_i - \epsilon} \cap F$ can be isotoped to be disjoint from $H_{v_i + \epsilon} \cap F$. Since some component of the former set bounds an sc-disc on one side of $H$ and some component of the latter set bounds an sc-disc on the other side of $H$, we have Conclusion \ref{it: sc-weakly red}  again.

\textbf{Case 4:} For every $i$, $I_i$ is labelled $\down$ and not $\up$ or for every $i$, $I_i$ is labelled $\up$ and not $\down$.

Without loss of generality, assume that each $I_i$ is labelled $\up$ and not $\down$. In particular, $I_1$ is labelled $\up$ and not $\down$. Fix $t \in I_1$ and consider $H_t$. Since $H$ does not bound a trivial v.p.-compressionbody to either side, the spine for  $(H_\down, T \cap H_\down)$ has an edge $e$. Since $I_i$ is below the lowest critical point for $h|_F$, the components of $F \cap (H_t)_\down$ intersecting $e$ are a regular neighborhood in $F$ of $F \cap e$. Let $D_\down$ be a meridian disc for $e$ with boundary in $H_t$ and which is disjoint from $F \cap (H_t)_\down$. Since $I_1$ is labelled $\up$, there is a component $\zeta \subset H_t \cap F$ such that $\zeta$ bounds an sc-disc $D_\up$ for $H_t$ in $(H_t)_\up$. The pair $\{D_\up, D_\down\}$ is then a weak reducing pair for $H_t$, giving Conclusion \eqref{it: sc-weakly red}.
\end{proof}

\begin{remark}
Observe that in Conclusion (3), we can only conclude that $H$ is sc-weakly reducible -- not that $H$ is c-weakly reducible. This arises in Case 4 of the proof, when we use an edge of the spine to produce an sc-disc. This is one reason for allowing semi-compressing and semi-cut discs in weak reducing pairs.
\end{remark}

\begin{corollary}\label{Vertical discs}
Suppose that $H \in \vpH(M,T)$ is connected and sc-strongly irreducible. If a component $S$ of $\boundary M$ is c-compressible, then the component of $(M,T) \setminus H$ containing $S$ is a trivial product compressionbody.
\end{corollary}
\begin{proof}
Let $F \subset (M,T)$ be a c-disc for $S$. Let $(C, T_C)$ and $(E, T_E)$ be the components of $(M,T)\setminus H$, with $S \subset \boundary_- C$. If $(E, T_E)$ is a trivial compressionbody, we may isotope $F$ out of $E$ to be contained in $C$. This contradicts the fact that $\boundary_- C$ is c-incompressible in $C$. Hence, $(E, T_E)$ is not a trivial product compressionbody.  Since $S \subset \boundary_- C$ is c-compressible, it either has positive genus or intersects $T$ at least 3 times. In particular, $(C, T_C)$ is not a trivial ball compressionbody. Suppose, for a contradiction, that $(C, T_C)$ is not a trivial product compressionbody. Then by Theorem \ref{Thm: Sweepout} $H$ can be isotoped transversally to $T$ such that after the isotopy one of the following holds:
\begin{enumerate}
\item $H \cap F = \nil$

\item $H \cap F \neq \nil$, every component of $H \cap F$ is essential in $F$ and no component of $H \cap F$ bounds an sc-disc for $H$.
\end{enumerate}

Since $\boundary_- C$ is c-incompressible in $C$, by Lemma \ref{Lem: Invariance}, the first conclusion cannot hold. Since no curve in a disc or once-punctured disc is essential, the second conclusion is also impossible. Thus, $(C, T_C)$ is a trivial product compressionbody.
\end{proof}

\begin{theorem}[Properties of locally thin surfaces]\label{Properties Locally Thin}
Suppose that $(M,T)$ is a (3-manifold, graph) pair, with $T$ irreducible. Let $\mc{H} \in \vpoH(M,T)$ be locally thin. Then the following hold:
\begin{enumerate}
\item $\mc{H}$ is reduced
\item Each component of $\mc{H}^+$ is sc-strongly irreducible in the complement of $\mc{H}^-$.
\item No component of $(M,T) \setminus \mc{H}$ is a trivial product compressionbody between $\mc{H}^-$ and $\mc{H}^+$. 
\item Every component of $\mc{H}^-$ is c-essential in $(M,T)$.
\item If $(M,T)$ is irreducible and if $\mc{H}$ contains a 2-sphere disjoint from $T$, then $T = \nil$ and $M = S^3$ or $M = B^3$.
\end{enumerate}
\end{theorem}
\begin{proof}
Without loss of generality, we may assume that $T$ has no vertices (drilling them out to turn them into components of $\boundary M$ if necessary). Conclusions (1) and (3) are immediate from the definition of locally thin. If some component of $\mc{H}^+$ is sc-weakly reducible in $(M,T)\setminus \mc{H}^-$, then, since $T$ is irreducible, we could perform an elementary thinning sequence, contradicting the definition of locally thin. Thus, (2) also holds.

Next we show that each component of $\mc{H}^-$ is c-incompressible. Suppose, therefore, that $S \subset \mc{H}^-$ is a thin surface. We first show that $S$ is c-incompressible and then that it is not $\boundary$-parallel.  Suppose that $S$ is c-compressible by a c-disc $D$. By an innermost disc argument, we may assume that no curve of $D \cap (\mc{H}^-\setminus S)$ is an essential curve in $\mc{H}^-$. By passing to an innermost disc, we may also assume that $D \cap (\mc{H}^-\setminus S) = \nil$. Let $(M_0, T_0)$ be the component of $(M,T) \setminus \mc{H}^-$ containing $D$. Let $H = \mc{H}^+ \cap M_0$ and recall that $H$ is connected. By Corollary \ref{Vertical discs} applied to $H$ in $(M_0, T_0)$, the v.p.-compressionbody between $S$ and $H$ is a trivial product compressionbody. This contradicts property (3) of locally thin multiple v.p.-bridge surfaces. Thus, each component of $\mc{H}^-$ is c-incompressible. 

We now show no sphere component of $\mc{H}^-$ bounds a 3-ball in $M \setminus T$. Suppose that $S \subset \mc{H}^-$ is such a sphere and let $B \subset M\setminus T$ be the 3-ball it bounds. By passing to an innermost such sphere, we may assume that no component of $\mc{H}^-$ in the interior of $B$ is a 2-sphere. If there is a component of $\mc{H}^-$ in the interior of $B$, that component would be compressible, a contradiction. Thus the intersection $H$ of $\mc{H}$ with the interior of $B$ is a component of $\mc{H}^+$. The surface $H$ is a Heegaard splitting of $B$. If $H$ is a sphere it is parallel to $S$, contradicting (3). If $H$ is not a sphere, then by \cite{Wald} it is stabilized, contradicting (1). Thus, each component of $\mc{H}^-$ is c-incompressible in $(M,T)$ and not a sphere bounding a 3-ball in $M\setminus T$. In particular, if $(M,T)$ is irreducible no component of $\mc{H}^-$ is a sphere disjoint from $T$. 

We now show that no component of $\mc{H}^-$ is $\boundary$-parallel. Since $T$ may be a graph and not simply a link, this does not follow immediately from our previous work. Suppose, to obtain a contradiction, that a component $F$ of $\mc{H}^-$ is boundary parallel in the exterior of $T$. An analysis (which we provide momentarily) of the proof of \cite[Theorem 9.3]{TT2} shows that $\mc{H}$ either has a perturbation or a generalized stabilization or is removable. We elaborate on this:

As in \cite[Lemma 3.3]{TT2}, since all components of $\mc{H}^-$ are c-incompressible, we may assume that the product region $W$ between $F$ and $\boundary (M\setminus T)$ has interior disjoint from $\mc{H}^-$. (That is, $F$ is innermost.) Observe that $W$ is a compressionbody with $F = \boundary_+ W$ and the component $H = \mc{H}^+ \cap W$ is a v.p.-bridge surface for $(W, T \cap W)$.

If there is a component of $T \cap W$ with both endpoints on $F$, then $F$ must be a 2-sphere and $W$ is a 3-ball with $T \cap W$ a $\boundary$-parallel arc. In this case, if $T \cap W$ is disjoint from $H$, then $H$ is a Heegaard surface for the solid torus obtained by drilling out the arc $T \cap W$. Since $H$ is not stabilized, it must then be a torus. In particular, since $W$ is a 3--ball, this implies that $H$ is meridionally stabilized, a contradiction. Thus, in particular, if $T \cap W$ has a component with both endpoints on $F$, then $H$ intersects each component of $T \cap W$. We now peform a trick to guarantee that this is also the case when $T \cap W$ does have a component with endpoints on $F$.

Let $(\ob{W}, \ob{T})$ be the result of removing from $(W, T \cap W)$ an open regular neighborhood of all edges of $T \cap W$ which are disjoint from $H$. (By our previous remark, all such edges have both endpoints on $\boundary W \setminus F$.) Then $\ob{T}$ is a 1--manifold properly embedded in $\ob{W}$ with no edge disjoint from $H$. 

If $\ob{T}$ has at least one edge, then by \cite[Theorem 3.5]{TT2} (which is a strengthening of \cite[Theorem 3.1]{TT1}) one of the following occurs:
\begin{enumerate}
\item[(i)] $H \in \vpH(\ob{W}, \ob{T})$ is stabilized, boundary-stabilized along $\boundary \ob{W} \setminus F$, perturbed or removable. 
\item[(ii)] $H$ is parallel to $F$ by an isotopy transverse to $T$.
\end{enumerate}

Consider possibility (i). If $H \in \vpH(\ob{W}, \ob{T})$ is stabilized, boundary-stabilized along $\boundary \ob{W}$, or removable then $H \in \vpH(W, T)$ would have a generalized stabilization or be removable with removing discs disjoint from the vertices of $T$, an impossibility. If $H \in \vpH(\ob{W}, \ob{T})$ is perturbed, $H \in \vpH(W, T)$ has a perturbation (since $\ob{T}$ is a 1-manifold), also an impossibility. Thus, (i) does not occur. Possibility (ii) does not occur since, none of the v.p.-compressionbodies of $(M,T) \setminus \mc{H}$ are trivial product compressionbodies adjacent to $\mc{H}^-$.

We may assume, therefore, that $\ob{T}$ has no edges. Then by \cite{ST-class products}, $H$ is either parallel to $F$ by an isotopy transverse to $T$ or is boundary-stabilized along $\boundary \ob{W}$. The former situation contradicts the assumption that $(M,T)\setminus \mc{H}$ contains no trivial product compressionbody adjacent to $F \subset \mc{H}^-$. In the latter situation, since $H$ is boundary-stabilized in $\ob{W}$, then it has a generalized stabilization as a surface in $\vpoH(W,T)$, contradicting the assumption that $\mc{H}$ is reduced. Thus, once again, $F$ is not boundary-parallel in $M \setminus T$. We have shown, therefore, that no component of $\mc{H}^-$ is $\boundary$-parallel and, thus, that each component of $\mc{H}^-$ is c-essential in $(M,T)$.

It remains to show that if $(M,T)$ is irreducible and if some component of $\mc{H}$ is a sphere disjoint from $T$, then $T = \nil$ and $M = B^3$ or $M = S^3$. Assume that $(M,T)$ is irreducible. We have already remarked that since each component of $\mc{H}^-$ is c-essential no component of $\mc{H}^-$ is a sphere disjoint from $T$. We now show that no component $H$ of $\mc{H}^+$ is a sphere disjoint from $T$, unless $T = \nil$ and $M$ is $S^3$ or $B^3$.  Suppose that there is such a component $H \subset \mc{H}^+$. Let $(C, T_C)$ and $(D, T_D)$ be the v.p.-compressiobodies on either side of $H$. By the definition of v.p.-compressionbody the surfaces $\boundary_- C$ and $\boundary_- D$ are the unions of spheres and $T_C$ and $T_D$ are the unions of ghost arcs. Consider the graphs $\Gamma_C = \boundary_- C \cup T_C$ and $\Gamma_D = \boundary_- D \cup T_D$ (thinking of the components of $\boundary_- C$ and $\boundary_- D$ as vertices of the graph.) By the definition of v.p.-compressionbody since $H$ is a sphere disjoint from $T$, the graphs $\Gamma_C$ and $\Gamma_D$ are the union of trees. If either $\Gamma_C$ or $\Gamma_D$ has an edge, then a leaf of $\Gamma_C$ or $\Gamma_D$ is a sphere intersecting $T$ exactly once. This contradicts the irreducibility of $(M,T)$. Consequently, both $T_C$ and $T_D$ are empty. Since no spherical component of $\mc{H}^-$ is disjoint from $T$, this implies that $\boundary_- C \cup \boundary_- D$ is a subset of $\boundary M$. Since $M\setminus T$ is irreducible, this implies that $\boundary_- C \cup \boundary_- D$ is either empty or a single sphere. Consequently, $M$ is either $S^3$ or $B^3$ and $T = \nil$. \end{proof}

\section{Decomposing spheres}\label{thin decomp spheres}

The goal of this section is to show that if we have a bridge surface for a composite knot or graph, we can untelescope it so that a summing sphere shows up as a thin level. 

We start with a simple observation (likely well-known) that compressing essential twice and thrice-punctured spheres results in a component which is still essential. The proof is straightforward and similar to that of Lemma \ref{spheres in v.p.-compressionbodies}, so we leave it to the reader. 

\begin{lemma}\label{lem:Always essential}
Assume that $(M,T)$ is a 3-manifold graph pair with $T$ irreducible. Suppose that $P \subset (M,T)$ is an essential sphere with $|P \cap T| \leq 3$. Let $P'$ be the result of compressing $P$ along an sc-disc $D$. Then at least one component of $P' \subset (M,T)$ is an essential sphere intersecting $T$ at most 3 times.
\end{lemma}

\begin{theorem}\label{thm:They are thin levels}
Suppose that $(M,T)$ is a (3-manifold, graph) pair with $T$ irreducible. Suppose that there is an essential sphere $P \subset (M,T)$ such that $|P \cap T| \leq 3$. If $\mc{H} \in \vpoH(M,T)$ is locally thin, then some component $F$ of $\mc{H}^-$ is an essential sphere with $|F \cap T| \leq 3$. Furthermore, $|F \cap T| \leq |P \cap T|$ and $F$ can be obtained from $P$ by a sequence of compressions using sc-discs.
\end{theorem}
\begin{proof}
Let $P \subset (M,T)$ be an essential sphere such that $|P \cap T| \leq 3$. As in the proof of Lemmas \ref{spheres in v.p.-compressionbodies} and \ref{lem:Always essential}, since $T$ is irreducible, if $P_0$ is a sphere resulting from an sc-compression of $P$, then $|P_0 \cap T| \leq |P \cap T|$ since $T$ is irreducible. 

Without loss of generality, we may assume that the given $P$ was chosen so that no sequence of isotopies and sc-compressions reduces $|P \cap \mc{H}^-|$. The intersection $P \cap \mc{H}^-$ consists of a (possibly empty) collection of circles. We show it is, in fact, empty.  Suppose, for a contradiction, that  $\gamma$ is a component of $|P \cap \mc{H}^-|$. Without loss of generality, we may suppose it is innermost on $P$. Let $D \subset P$ be the unpunctured disc or once-punctured disc which it bounds. Since $\mc{H}^-$ is c-incompressible, $\gamma$ must bound a zero or once-punctured disc $E$ in $\mc{H}^-$. Thus, if $|P \cap \mc{H}^-| \neq 0$, then there is a component of the intersection which is inessential in $\mc{H}^-$.

Let $\zeta \subset P \cap \mc{H}^-$ be a component which is inessential in $\mc{H}^-$ and which, out of all such curves, is innermost in $\mc{H}^-$. Let $E \subset \mc{H}^-$ be the unpunctured or once-punctured disc it bounds. Observe that $\zeta$ also bounds a  zero or once-punctured disc on $P$. If $E$ is not an sc-disc for $P$, then we can isotope $P$ to reduce $|P \cap \mc{H}^-|$, contradicting our choice of $P$. Thus, $E$ is an sc-disc. By Lemma \ref{lem:Always essential}, compressing $P$ along $E$ creates two spheres, at least one of which intersects $T$ no more than 3 times and is essential in the exterior of $T$. Since this component intersects $\mc{H}^-$ fewer times than does $P$, we have contradicted our choice of $P$. Hence $P \cap \mc{H}^- = \nil$.

We now consider intersections between $P$ and $\mc{H}^+$. Since $P$ is disjoint from $\mc{H}^-$, we may apply Theorem \ref{Thm: Sweepout} to the component $(W, T_W)$ of $(M,T)\setminus \mc{H}^-$ containing $P$. We apply the theorem with $H = \mc{H}^+ \cap W$ and $F = P$. If some component of $\mc{H}^-$ is a once-punctured sphere, we are done, so assume that no component of $\mc{H}^-$ is a once-punctured sphere. By cutting open along $\mc{H}^-$ and replacing $(M,T)$ with the component containing $P$, we may assume that $\mc{H}^- = \nil$ and that $H = \mc{H}$ is connected. Apply Theorem \ref{Thm: Sweepout} to $P$ (in place of $F$) to see that we can isotope $H$ transversally to $T$ in $M \setminus \mc{H}^-$ so that one of the following occurs:
\begin{enumerate}
\item $H \cap P = \nil$.
\item $H \cap P$ is a non-empty collection of curves, each of which is essential in $P$.
\item $H$ is sc-weakly reducible.
\end{enumerate}
Since $\mc{H}$ is locally thin in $\vpoH(M,T)$, (3) does not occur. Since $P$ contains no essential curves, (2) does not occur. Thus, $H \cap P = \nil$. 

Let $(C, T_C)$ be the component of $M \setminus \mc{H}$ containing $P$. By Lemma \ref{spheres in v.p.-compressionbodies}, after some sc-compressions, $P$ is parallel is a component of $\mc{H}^- \cup \boundary M$. By Lemma \ref{lem:Always essential}, $P$ is parallel to a component of $\mc{H}^-$ and we are done.
 \end{proof}


\bib{Bachman}{article}{
  author={Bachman, David},
  title={Thin position with respect to a Heegaard surface},
  eprint={http://pzacad.pitzer.edu/~dbachman/Papers/tnbglex.pdf},
}

\bib{CG}{article}{
  author={Casson, A. J.},
  author={Gordon, C. McA.},
  title={Reducing Heegaard splittings},
  journal={Topology Appl.},
  volume={27},
  date={1987},
  number={3},
  pages={275--283},
  issn={0166-8641},
  review={\MR {918537}},
  doi={10.1016/0166-8641(87)90092-7},
}

\bib{G3}{article}{
  author={Gabai, David},
  title={Foliations and the topology of $3$-manifolds. III},
  journal={J. Differential Geom.},
  volume={26},
  date={1987},
  number={3},
  pages={479--536},
  issn={0022-040X},
  review={\MR {910018 (89a:57014b)}},
}

\bib{GST}{article}{
  author={Goda, Hiroshi},
  author={Scharlemann, Martin},
  author={Thompson, Abigail},
  title={Levelling an unknotting tunnel},
  journal={Geom. Topol.},
  volume={4},
  date={2000},
  pages={243--275 (electronic)},
  issn={1465-3060},
  review={\MR {1778174}},
  doi={10.2140/gt.2000.4.243},
}

\bib{GL}{article}{
  author={Gordon, C. McA.},
  author={Luecke, J.},
  title={Knots are determined by their complements},
  journal={J. Amer. Math. Soc.},
  volume={2},
  date={1989},
  number={2},
  pages={371--415},
  issn={0894-0347},
  review={\MR {965210}},
  doi={10.2307/1990979},
}

\bib{HS01}{article}{
  author={Hayashi, Chuichiro},
  author={Shimokawa, Koya},
  title={Thin position of a pair (3-manifold, 1-submanifold)},
  journal={Pacific J. Math.},
  volume={197},
  date={2001},
  number={2},
  pages={301--324},
  issn={0030-8730},
  review={\MR {1815259 (2002b:57020)}},
  doi={10.2140/pjm.2001.197.301},
}

\bib{HRS}{article}{
  author={Howards, Hugh},
  author={Rieck, Yo'av},
  author={Schultens, Jennifer},
  title={Thin position for knots and 3-manifolds: a unified approach},
  conference={ title={Workshop on Heegaard Splittings}, },
  book={ series={Geom. Topol. Monogr.}, volume={12}, publisher={Geom. Topol. Publ., Coventry}, },
  date={2007},
  pages={89--120},
  review={\MR {2408244}},
  doi={10.2140/gtm.2007.12.89},
}

\bib{Johnson}{article}{
  author={Johnson, Jesse},
  title={Calculating isotopy classes of Heegaard splittings},
  eprint={arXiv:1004.4669},
}

\bib{Fabiola}{article}{
  author={Manjarrez-Guti{\'e}rrez, Fabiola},
  title={Circular thin position for knots in $S\sp 3$},
  journal={Algebr. Geom. Topol.},
  volume={9},
  date={2009},
  number={1},
  pages={429--454},
  issn={1472-2747},
  review={\MR {2482085}},
  doi={10.2140/agt.2009.9.429},
}

\bib{RS}{article}{
  author={Rieck, Yo'av},
  author={Sedgwick, Eric},
  title={Thin position for a connected sum of small knots},
  journal={Algebr. Geom. Topol.},
  volume={2},
  date={2002},
  pages={297--309 (electronic)},
  issn={1472-2747},
  review={\MR {1917054 (2003d:57021)}},
  doi={10.2140/agt.2002.2.297},
}

\bib{RubSch}{article}{
  author={Rubinstein, Hyam},
  author={Scharlemann, Martin},
  title={Comparing Heegaard splittings of non-Haken $3$-manifolds},
  journal={Topology},
  volume={35},
  date={1996},
  number={4},
  pages={1005--1026},
  issn={0040-9383},
  review={\MR {1404921}},
  doi={10.1016/0040-9383(95)00055-0},
}

\bib{Scharlemann-Survey}{article}{
  author={Scharlemann, Martin},
  title={Heegaard splittings of compact 3-manifolds},
  conference={ title={Handbook of geometric topology}, },
  book={ publisher={North-Holland, Amsterdam}, },
  date={2002},
  pages={921--953},
  review={\MR {1886684 (2002m:57027)}},
}

\bib{Scharlemann-Schultens-JSJ}{article}{
  author={Scharlemann, Martin},
  author={Schultens, Jennifer},
  title={Comparing Heegaard and JSJ structures of orientable 3-manifolds},
  journal={Trans. Amer. Math. Soc.},
  volume={353},
  date={2001},
  number={2},
  pages={557--584 (electronic)},
  issn={0002-9947},
  review={\MR {1804508}},
  doi={10.1090/S0002-9947-00-02654-4},
}

\bib{ST-thin}{article}{
  author={Scharlemann, Martin},
  author={Thompson, Abigail},
  title={Thin position for $3$-manifolds},
  conference={ title={Geometric topology}, address={Haifa}, date={1992}, },
  book={ series={Contemp. Math.}, volume={164}, publisher={Amer. Math. Soc.}, place={Providence, RI}, },
  date={1994},
  pages={231--238},
  review={\MR {1282766 (95e:57032)}},
}

\bib{ST-graphthin}{article}{
  author={Scharlemann, Martin},
  author={Thompson, Abigail},
  title={Thin position and Heegaard splittings of the $3$-sphere},
  journal={J. Differential Geom.},
  volume={39},
  date={1994},
  number={2},
  pages={343--357},
  issn={0022-040X},
  review={\MR {1267894}},
}

\bib{ST-classproducts}{article}{
  author={Scharlemann, Martin},
  author={Thompson, Abigail},
  title={Heegaard splittings of $({\rm surface})\times I$ are standard},
  journal={Math. Ann.},
  volume={295},
  date={1993},
  number={3},
  pages={549--564},
  issn={0025-5831},
  review={\MR {1204837}},
  doi={10.1007/BF01444902},
}

\bib{STo}{article}{
  author={Scharlemann, Martin},
  author={Tomova, Maggy},
  title={Uniqueness of bridge surfaces for 2-bridge knots},
  journal={Math. Proc. Cambridge Philos. Soc.},
  volume={144},
  date={2008},
  number={3},
  pages={639--650},
  issn={0305-0041},
  review={\MR {2418708 (2009c:57020)}},
  doi={10.1017/S0305004107000977},
}

\bib{TT1}{article}{
  author={Taylor, Scott A.},
  author={Tomova, Maggy},
  title={Heegaard surfaces for certain graphs in compressionbodies},
  journal={Rev. Mat. Complut.},
  volume={25},
  date={2012},
  number={2},
  pages={511--555},
  issn={1139-1138},
  review={\MR {2931424}},
  doi={10.1007/s13163-011-0075-6},
}

\bib{TT2}{article}{
  author={Taylor, Scott A.},
  author={Tomova, Maggy},
  title={C-essential surfaces in (3-manifold, graph) pairs},
  journal={Comm. Anal. Geom.},
  volume={21},
  date={2013},
  number={2},
  pages={295--330},
  issn={1019-8385},
  review={\MR {3043748}},
  doi={10.4310/CAG.2013.v21.n2.a2},
}

\bib{TT4}{article}{
  author={Taylor, Scott A.},
  author={Tomova, Maggy},
  title={Additive invariants for links and graphs in 3-manifolds},
  status={Accepted by Geometry \& Topology},
  eprint={arXiv:1606.03408},
}

\bib{Thompson}{article}{
  author={Thompson, Abigail},
  title={Thin position and bridge number for knots in the $3$-sphere},
  journal={Topology},
  volume={36},
  date={1997},
  number={2},
  pages={505--507},
  issn={0040-9383},
  review={\MR {1415602}},
  doi={10.1016/0040-9383(96)00010-9},
}

\bib{Wald}{article}{
  author={Waldhausen, Friedhelm},
  title={Heegaard-Zerlegungen der $3$-Sph\"are},
  language={German},
  journal={Topology},
  volume={7},
  date={1968},
  pages={195--203},
  issn={0040-9383},
  review={\MR {0227992}},
}



\begin{bibdiv}
\begin{biblist}
\bibselect{AdditiveInvariantsBib2}
\end{biblist}
\end{bibdiv}
\end{document}